\documentclass[12pt,amstex]{amsart}
\usepackage{stmaryrd}

\usepackage{epsfig}
\usepackage{amsmath}
\usepackage{amssymb}
\usepackage{amscd}
\usepackage{graphicx}

\usepackage{pstricks}

\usepackage{tocvsec2}

\usepackage{hyperref}

\input xy
\xyoption{all}

\newtheorem{thm}{Theorem}[section]
\newtheorem{lem}[thm]{Lemma}

\newtheorem{prop}[thm]{Proposition}
\newtheorem{ques}[thm]{Question}
\newtheorem{cor}[thm]{Corollary}
\theoremstyle{definition}
\newtheorem{de}[thm]{Definition}

\theoremstyle{remark}
\newtheorem{rem}[thm]{Remark}

\numberwithin{equation}{section}

\begin{document}

\title{Measure complexity and M\"{o}bius disjointness}
\author{Wen Huang}
\author{Zhiren Wang}
\author{Xiangdong Ye}

\address[W.~Huang]{Wu Wen-Tsun Key Laboratory of Mathematics, USTC, Chinese Academy of Sciences,
Department of Mathematics, University of Science and Technology of China,
Hefei, Anhui, 230026, P.R. China} \email{wenh@mail.ustc.edu.cn}

\address[Z.~Wang]{Department of Mathematics, Pennsylvania State University, University Park,
PA 16802, USA}
\email{zhirenw@psu.edu}
\address[X.~Ye]{Wu Wen-Tsun Key Laboratory of Mathematics, USTC, Chinese Academy of Sciences,
Department of Mathematics, University of Science and Technology of China,
Hefei, Anhui, 230026, P.R. China} \email{yexd@ustc.edu.cn}
\date{Feb. 1, 2017}



\begin{abstract}
{In this paper, the notion of measure complexity is introduced for a topological dynamical system and it is shown that Sarnak's M\"{o}bius
disjointness conjecture holds for any system for which every invariant Borel probability measure has sub-polynomial measure complexity.

Moreover, it is proved that the following classes of topological dynamical systems $(X,T)$ meet this condition and hence satisfy Sarnak's conjecture:
(1) Each invariant Borel probability measure of $T$ has discrete spectrum.
(2) $T$ is a homotopically trivial $C^\infty$
skew product system on $\mathbb{T}^2$ over an irrational rotation of the circle.
Combining this with the previous results it implies that the M\"{o}bius disjointness conjecture holds for any $C^\infty$ skew product system on $\mathbb{T}^2$.
(3) $T$ is a continuous skew product map of the form $(ag,y+h(g))$  on $G\times \mathbb{T}^1$  over a minimal rotation
of the compact metric abelian group $G$ and $T$  preserves a measurable section.
(4) $T$ is a tame  system.
~}
\end{abstract}

\maketitle

\markboth{Measure complexity and M\"{o}bius disjointness}{}




\section{Introduction}

The M\"{o}bius function $\mu: \mathbb{N}\rightarrow \{-1,0,1\}$ is defined by
$\mu(1)=1$ and
\begin{equation}\label{M-function}
  \mu(n)=\left\{
           \begin{array}{ll}
             (-1)^k & \hbox{if $n$ is a product of $k$ distinct primes;} \\
             0 & \hbox{otherwise.}
           \end{array}
         \right.
\end{equation}

Let $(X,T)$ be a (topological) dynamical system (for short t.d.s.), namely $X$ is a compact metrizable space and $T : X \rightarrow X$
a homeomorphism. We say a sequence $\xi$ is {\em realized} in $(X,T)$ if there is an $f\in C(X)$ and an $x\in X$
such that $\xi(n) = f(T^nx)$ for any $n\in\mathbb{N}$. A sequence $\xi$ is called {\em deterministic} if it is realized in a
system with zero topological entropy. Here is the well-known conjecture by Sarnak \cite{Sar}:

\medskip

\noindent {\bf M\"{o}bius Disjointness Conjecture:}\ {\em
The M\"{o}bius function $\mu$ is linearly disjoint from any deterministic sequence $\xi$. That is,
\begin{equation}\label{Sarnak}
  \lim_{N\rightarrow \infty}\frac{1}{N}\sum_{n=1}^N\mu(n)\xi(n)=0.
\end{equation}
}

It is known that by Green and Tao \cite{GT} that nilsystems satisfy the conjecture.
We refer to \cite{MR,G,B1,B2,BSZ,ALR,MMR,LS,KPL,P,ALR1,V1,FJ,W1,HLSY,AKLR,Wei,HWZ}
for the progress on this conjecture.

By the work of El Abdalaoui, Lema\'{n}cyzk and de la Rue \cite{ALR1},
M\"{o}bius disjointness conjecture holds for any  topological model of an ergodic system with
irrational discrete spectrum (in fact  in \cite{ALR1} M\"{o}bius disjointness conjecture  is proved for any  topological model of
a totally ergodic system with {\it quasi-discrete spectrum}. We note that any ergodic automorphism
with irrational discrete spectrum has quasi-discrete spectrum and totally ergodic).
Strengthening the result, recently it is shown that M\"{o}bius disjointness conjecture holds for any topological
model of an ergodic system with discrete spectrum by Huang, Wang and Zhang \cite{HWZ}. Note that when we speak about the topological model, we mean it is
uniquely ergodic. So a natural question is if the conjecture holds
for a t.d.s. with the property that each invariant measure has discrete spectrum. In the process to study the question,
we not only get an affirmative answer to the question, but also obtain other related results which solve some open questions.

As the main result of the paper we  provide a criterion for a t.d.s. satisfying the required disjointness condition in
Sarnak's M\"{o}bius disjointness conjecture using the notion of the measure complexity we now
introduce. Precisely, let $(X,T)$ be a t.d.s. with a metric $d$ and let ${\mathcal{M}}(X,T)$ be the set of all
$T$-invariant Borel probability measures on $X$. For $\rho\in {\mathcal{M}}(X,T)$ and any $n\in \mathbb{N}$, we consider the metric
$$\overline{d}_n(x,y)=\frac{1}{n}\sum_{i=0}^{n-1}d(T^ix,T^iy)$$
for any $x,y\in X$. For $\epsilon>0$, let
$$S_n(d,\rho,\epsilon)=\min \{ m\in \mathbb{N}:\exists x_1,x_2,\cdots,x_m
\text{ s.t. } \rho\big(\bigcup_{i=1}^m B_{\overline{d}_n}(x_i,\epsilon)\big)>1-\epsilon\},$$
where $B_{\overline{d}_n}(x,\epsilon):=\{y\in X: \overline{d}_n(x,y)<\epsilon\}$
for any $x\in X$.

Let $U(n):\mathbb{N}\rightarrow [1,+\infty)$ be an increasing sequence with $\lim_{n\rightarrow +\infty} U(n)=+\infty$.
Following  the idea of  Ferenczi \cite{Fer}, we say the measure complexity of $(X,d,T,\rho)$ is weaker than $U(n)$,
if $\liminf_{n\rightarrow +\infty} \frac{S_n(d,\rho,\epsilon)}{U(n)}=0$ for any $\epsilon>0$. By Proposition \ref{prop-0},
the measure complexity of $(X,d,T,\rho)$ is weaker than $U(n)$ if and only if the measure complexity of
$(X,d',T,\rho)$ is also weaker than $U(n)$ for any compatible metric $d'$ on $X$.
Thus we can simply say the measure complexity of $(X,T,\rho)$ is weaker than $U(n)$.

We say the measure complexity of $(X,T,\rho)$ is {\it sub-polynomial}, if the measure complexity of $(X,T,\rho)$ is
weaker than $U_\tau(n)=n^\tau$ for any $\tau>0$. Our main result is the following.

\begin{thm}  \label{main-result1}  Let $(X,T)$ be a t.d.s. and the measure complexity of $(X,T,\rho)$  be sub-polynonimal
for any $\rho\in {\mathcal{M}}(X,T)$. Then the  M\"{o}bius disjointness conjecture holds.
\end{thm}

In fact Theorem \ref{main-result1} is proved via the following equivalent form.

\medskip
\noindent {\bf Theorem 1.1'.} {\it  Let $(X,T)$ be a t.d.s. with $x\in X$ and $\{ N_1<N_2<N_3<\cdots\} \subseteq \mathbb{N}$ such that
the sequence $\frac{1}{N_i}\sum_{n=1}^{N_i} \delta_{T^nx}$ weakly$^*$ converges to a $\rho\in  {\mathcal{M}}(X,T)$.
Suppose that the measure complexity of $(X,T,\rho)$  is sub-polynonimal. Then
$$\lim_{i\rightarrow +\infty} \frac{1}{N_i} \sum_{n=1}^{N_i} \mu(n) f(T^nx)=0$$
for any $f\in C(X)$.}

\medskip

As applications of Theorem \ref{main-result1}, we consider the following classes of t.d.s.

\subsection{Systems with discrete spectrum}

Let $(X, T )$ be a t.d.s, $\mathcal{B}_X$ be the Borel $\sigma$-algebra of $X$ and $\rho \in {\mathcal{M}}(X,T)$.
An eigenfunction for $T$ is some non-zero
function $f\in L^2(X,\mathcal{B}_X,\rho) = L^2(\rho)$ such that $Uf = f\circ T =\lambda f$ for some $\lambda \in \mathbb{C}$. $\lambda$ is
called the eigenvalue corresponding to $f$. It is easy to see every eigenvalue has norm one,
that is $|\lambda| = 1$. If $f\in L^2(\rho)$ is an eigenfunction, then $\text{cl}\{U^nf : n\in \mathbb{Z}\}$ is a compact
subset of $L^2(\rho)$. Generally, we say $f$ is almost periodic if $\text{cl}\{U^nf : n\in \mathbb{Z}\}$ is compact
in $L^2(\rho)$. It is well known that the set of all bounded almost periodic functions forms a
$U$-invariant and conjugation-invariant subalgebra of $L^2(\rho)$ (denoted by $A_c$). The set of all
almost periodic functions is just the closure of $A_c$ (denoted by $H_c$), and is also spanned by
the set of eigenfunctions. $T$ is said  to have {\it discrete spectrum} if $L^2(\rho)$
is spanned by the set of eigenfunctions, that is $H_c = L^2(X,\mathcal{B}_X,\rho)$.
If $\rho$ has discrete spectrum, then the measure complexity of $(X,T,\rho)$ is sub-polynomial
(see Proposition \ref{prop-1}).
Thus the following result is a direct corollary of Theorem \ref{main-result1} and Proposition \ref{prop-1}.


\begin{thm}  \label{main-result2} Let $(X,T)$ be a t.d.s. such that each $\rho\in {\mathcal{M}}(X,T)$ has discrete spectrum, then
M\"{o}bius disjointness conjecture holds for $(X,T)$.
\end{thm}

\subsection{$C^\infty$-skew product on $\mathbb{T}^2$}

Let $T$ be a skew product map on $\mathbb{T}^2$ over a  rotation of the circle. That is,
\begin{align*}
T(x, y) = (x+\alpha, y + h(x)),
\end{align*}
where $h : \mathbb{T}^1 \rightarrow \mathbb{T}^1$ is continuous and $\alpha\in [0,1)$.

\begin{thm} \label{thm-sub-polynomial} If  $\alpha$ is irrational and $h$ is a homotopically trivial $C^\infty$-function, then the measure complexity of every invariant Borel probability measure of $(\mathbb{T}^2,T)$ is sub-polynomial.
\end{thm}
Liu and Sarnak \cite{LS} showed that if $\alpha$ is rational, then M\"{o}bius disjointness conjecture holds
for $T$. By \cite[Remark 2.5.7]{KPL} or \cite[Corollary 2.6]{W1}, if $h$ is a Lipschtiz continuous map and $h$ is not homotopically trivial, then
M\"{o}bius disjointness conjecture holds
for $T$.  Thus combining these results with  Theorems \ref{main-result1} and \ref{thm-sub-polynomial}, we have the following
result which was first known by Tao \cite{Tao} using a different approach (See also \cite{LS,W1} for the analytic case. The proofs in \cite{LS,W1} and that of Tao seperate two cases which are treated respectively by the K\'atai-Bourgain-Sarnak-Ziegler criterion and the bound on short interval averages of multiplicative functions of Matom\"aki-Radziwi\l{}\l{}-Tao, while our result only relies on the bound of  Matom\"aki-Radziwi\l{}\l{}-Tao).

\begin{cor} \label{thm-cinfty} If $h$ is $C^\infty$, then for all $(x, y) \in \mathbb{T}^2$, and all continuous
functions $f\in C(\mathbb{T}^2)$,
\begin{align*}
\frac{1}{N}\sum_{n=1}^N \mu(n)f(T^n(x, y))\rightarrow 0.
\end{align*}
as $N\rightarrow +\infty$.
\end{cor}

\subsection{The non-uniquely ergodic skew product}
Let $T$ be a skew product map on $G\times \mathbb{T}^1$ over a minimal rotation of the compact metrizable
abelian group $G$ with Haar measure $m_G$. That is,
\begin{align}\label{eee-2}
T(g, y) = (a g, y + h(g)),
\end{align}
where $h : G \rightarrow \mathbb{T}^1$ is continuous, and $a\in G$ is such that $S_a: G\rightarrow G, g\mapsto ag$ is minimal (this is equivalent to
say that $\{a^n:n\in \mathbb{Z}\}$ is dense in $G$).

A {\it measurable invariant section} of $T$ is a graph $(g, \phi(g))$, where $\phi : G\rightarrow \mathbb{T}^1$
is a Borel-measurable map, such that $T(g, \phi(g))$ is still in the graph for $m_G$-a.e.
every $g$.

\begin{thm} \label{thm-4} Assume that $T$ preserves a measurable invariant section. Then for all $(g, y) \in G\times \mathbb{T}^1$, and all continuous
functions $f\in C(G\times \mathbb{T}^1)$,
\begin{align}\label{1}
\frac{1}{N}\sum_{n=1}^N \mu(n)f(T^n(g, y))\rightarrow 0.
\end{align}
as $N\rightarrow +\infty$.
\end{thm}
One important feature of Theorem  \ref{thm-4} is that it holds for all compact metrizable abelian groups $G$ and all continuous function $h$, without
assuming $G=\mathbb{T}^1$ and any smooth condition for $h$. By a dichotomy of Furstenberg \cite[Lemma 2.1]{F}, if a map $T$ of the
form \eqref{eee-2} is not uniquely ergodic, then for some positive integer $\xi$, the
equation $\phi(ag)-\phi(g) = \xi h(g)$ has a measurable solution $\phi: G\rightarrow \mathbb{T}^1$.
Let $\pi_\xi : G\times \mathbb{T}^1\rightarrow G\times \mathbb{T}^1$ be the $\xi$-to-one projection, i.e. $\pi_\xi(g, y) = \pi(g, \xi y)$. Then the
transform $T_\xi (g, y) = (ag, y+\xi \phi(g))$ is a topological factor of $T$ through $\pi$,
in other words, $\pi_\xi\circ T= T_\xi\circ \pi_\xi$. One can easily check that the graph $(g, \phi(g))$
is a measurable invariant section for $T_\xi$. Hence Theorem \ref{thm-4} implies:
\begin{cor}
Suppose $T$ is not uniquely ergodic, then for some $\xi\in\mathbb{N}$, the  $\xi$-to-one topological factor $(G\times \mathbb{T}^1,T_\xi)$ of $(G\times \mathbb{T}^1,T)$ satisfies
M\"{o}bius disjointness conjecture.
\end{cor}

\subsection{$K(\mathbb{Z})$ introduced by Veech}
The class $K(\mathbb{Z})$ is introduced by Veech in \cite{V}.
Given $f\in \ell^\infty(\mathbb{Z})$, let $B_f$ denote the smallest translation invariant $*$-subalgebra
of $\ell^\infty(\mathbb{Z})$ that contains $f$ and the constants. The maximal ideal space, $X(f)$,
of $B_f$ is compact, metrizable and contains an image of $\mathbb{Z}$ as a dense subset. Translation by one determines a homeomorphism, $T$, of
$X(f)$. $X(f)$ may be identified as the set of all pointwise limits of sequences of translates of $f$.

Recall from \cite{V} that $K(\mathbb{Z})$ is the set of $f \in  \ell^\infty(\mathbb{Z})$ such that $X(f)$,
which may be identified naturally with a bounded weak$^*$ closed set in
$\ell^\infty(\mathbb{Z})$, is separable in the norm topology. $K(\mathbb{Z})$ contains the Eberlein algebra, $\mathcal{W}(\mathbb{Z})$, of weakly almost periodic functions. In Proposition 2.3 and Theorem 1.4 of \cite{V1}, Veech proved that
$$\lim_{N\rightarrow +\infty} \frac{1}{N}\sum_{n=1}^N \mu(n)f(n)=0 \text{ for }f\in \mathcal{W}(\mathbb{Z}),$$
and each $f \in K(\mathbb{Z})$ is (strongly) deterministic, that is, the topological entropy of $(X_f,T)$ is zero. In \cite{V1} Veech asked if
$$\lim_{N\rightarrow +\infty} \frac{1}{N}\sum_{n=1}^N \mu(n)f(n)=0.
$$
 holds for $f\in K(\mathbb{Z})\setminus \mathcal{W}(\mathbb{Z})$.
In this paper we affirmatively answer the question by proving that for $f \in K(\mathbb{Z})$, every $T$-invariant Borel probability measure of $(X_f,T)$
has discrete spectrum (see Proposition \ref{prop-2} in Section \ref{section-proof1}). Thus, we have

\begin{thm}\label{thm-5} If $f\in K(\mathbb{Z})$, then
\begin{align}\label{5}
\lim_{N\rightarrow +\infty} \frac{1}{N}\sum_{n=1}^N \mu(n)f(n)=0.
\end{align}
\end{thm}

A system related to $K(\mathbb{Z})$ is the tame system introduced by E. Glasner in \cite{Gl}.
The enveloping (or Ellis) semigroup $E(X, T)$ of a dynamical system $(X, T)$ is defined
as the closure in $X^X$ (with its compact, usually non-metrizable, pointwise convergence
topology) of the set $\{T^n : X\rightarrow X\}_{n\in \mathbb{Z}}$ considered as a subset of $X^X$.

In \cite{K}, K\"{o}hler pointed out the relevance of a theorem of Bourgain, Fremlin and
Talagrand \cite{BFT} to the study of enveloping semigroups. In \cite{GM}, Glasner and Megrelishvili
obtained a dynamical Bourgain-Fremlin-Talagrand (BFT) dichotomy: the enveloping semigroup of a dynamical
system is either very large and contains a topological copy of $\beta{\mathbb{N}}$, or it is a ��tame��
topological space whose topology is determined by the convergence of sequences. In the latter case,
Glasner calls the system {\it tame} \cite{Gl}. Examples of tame dynamical systems include metric minimal equicontinuous
systems, topologically transitive weakly almost periodic (WAP) systems (see \cite{GM}),
topologically transitive hereditarily non-sensitive (HNS) systems (see \cite{GM}) and null
systems (see \cite{HLSY03, KL}).

It is shown in \cite[Theorem 5.2]{H} that every $\rho\in {\mathcal{M}}(X,T)$ of a tame
system $(X,T)$ has discrete spectrum. Hence by Theorems \ref{main-result1} and \ref{main-result2} we have the following result.

\begin{thm} M\"{o}bius disjointness conjecture holds for a tame system.
\end{thm}

\medskip

\noindent{\bf Acknowledgements.} {We thank El Abdalaoui for bringing our attention to the work of
\cite{V1}, and for informing us the open question of M\"{o}bius Orthogonality for $K(\mathbb{Z})$.
W. Huang and X. Ye are supported by NNSF of China (11225105, 11431012, 11571335). Z. Wang
was supported by NSF (DMS-1501095).}

\section{Some basic properties of the induced metric}

In this section we first show that the property that
the measure complexity of a $\rho\in {\mathcal{M}}(X,T)$ is sub-polynomial is independent of the metrics. In fact we shall prove more than that,
for the details see Proposition \ref{prop-0}. Then we will discuss how  entropy is related to the metric $d_n$ induced from a metric $d$.

\subsection{Independence of the metrics}
Let $(X,T)$ and $(Y,S)$ be two t.d.s. with the metrics $d$ and $d'$ respectively,
and let $\rho\in {\mathcal{M}}(X,T)$ and $\nu\in {\mathcal{M}}(Y,S)$.  We say $(X,\mathcal{B}_X,T,\rho)$ is
{\it measurably isomorphic} to $(Y,\mathcal{B}_Y,S,\nu)$, if there are $X'\in \mathcal{B}_X, Y' \in \mathcal{B}_Y$
with $\rho(X') = 1$, $\nu(Y') = 1$, $TX'\subseteq X'$, $SY'\subseteq Y'$,
and an invertible measure-preserving map $\phi : X' \rightarrow Y'$ with $\phi \circ T(x) =
S \circ \phi (x)$ for all $x\in X'$.

The following is a fork fact, we state it as a lemma since we use it several times in the sequel.
\begin{lem}\label{severaltimes} Let $(X,T)$ be a t.d.s., $\epsilon>0$  and $K\in \mathcal{B}_X$ with $\rho(K)>1-\epsilon^2$, where
$\rho\in\mathcal{M}(X,T)$.
For $n\in\mathbb{N}$ and $x\in X$, let $E(x)=\{i\ge 0: T^ix\in K\}$ and $ E_n=\{x\in X:\frac{|E(x)\cap [0,n-1]|}{n}\le 1-\epsilon\}.$ Then $\rho(E_n)<\epsilon$.
\end{lem}
\begin{proof}For  $x\in X$, let
$E(x)=\{i\ge 0: T^ix\in K\}$ and
$ E_n=\{x\in X:\frac{|E(x)\cap [0,n-1]|}{n}\le 1-\epsilon\}.$
Note that
\begin{align*}
\int_X \frac{|E(x)\cap [0,n-1]|}{n} d\rho(x)=\int_X \frac{1}{n}\sum_{i=0}^{n-1} 1_{K}(T^ix) d\rho(x)=\rho(K)>1-\epsilon^2.
\end{align*}
We have
$$(1-\epsilon)\rho(E_n)+1-\rho(E_n)\ge\int_X \frac{|E(x)\cap [0,n-1]|}{n} d\rho(x)>1-\epsilon^2$$
which implies that $\rho(E_n)<\epsilon$.
\end{proof}

\begin{prop}\label{prop-0} Assume $(X,\mathcal{B}_X,T,\rho)$ is measurably isomorphic to $(Y,\mathcal{B}_Y,S,\nu)$,
and $U(n):\mathbb{N}\rightarrow [1,+\infty)$ is an increasing sequence with $\lim_{n\rightarrow +\infty} U(n)=+\infty$.
Then the measure complexity of $(X,d,T,\rho)$  is
weaker than $U(n)$ if and only if the measure complexity of $(Y,d',S,\nu)$ is weaker than $U(n)$.
\end{prop}
\begin{proof} It suffices to show that if the measure complexity of $(X,d,T,\rho)$  is
weaker than $U(n)$, then  the measure complexity of $(Y,d',S,\nu)$  is weaker than $U(n)$.

Suppose that the measure complexity of $(X,d,T,\rho)$  is
weaker than $U(n)$.
Since $(X,\mathcal{B}_X,T,\rho)$ is measurably isomorphic to $(Y,\mathcal{B}_Y,S,\nu)$, there exist
$X'\in \mathcal{B}_X, Y' \in \mathcal{B}_Y$ with $\rho(X') = 1$, $\nu(Y') = 1$, $TX'\subseteq X'$, $SY'\subseteq Y'$,
and an invertible measure-preserving map $\phi : X' \rightarrow Y'$ with $\phi \circ T(x) =
S \circ \phi (x)$ for all $x\in X'$.

Fix $\epsilon>0$. By Lusin's Theorem there exists a compact subset $A$ of $X'$ such that
$\rho(A)>1-(\frac{\epsilon}{4})^2$ and  $\phi|_A$ is a continuous function. Assume that $\text{diam}(Y)\le 1$. Choose $\delta\in (0,\epsilon)$ such that
$\sqrt{\delta}\cdot \text{diam}(Y)+ \frac{\epsilon}{2}\cdot  \text{diam}(Y)+  \frac{\epsilon}{3}<\epsilon$ and
\begin{equation}\label{hz-1}
d'(\phi(x),\phi(y))<\frac{\epsilon}{3}\ \text{for any}\ x,y\in A\ \text{with}\ d(x,y)<\sqrt{\delta}.
\end{equation}
Now we are going to show that $S_n(d,\rho,\frac{\delta}{2})\ge S_n(d',\nu,\epsilon)$ for all $n\in \mathbb{N}$.

Fix $n\in \mathbb{N}$. For $x\in A$ let $E(x)=\{i\ge 0: T^ix\in A\}$ and let
$E_n=\{x\in A:\frac{|E(x)\cap [0,n-1]|}{n}\le 1-\epsilon/4\}.$ By Lemma \ref{severaltimes}, $\rho(E_n)< \epsilon/4$.

For $x,y\in A$, if $\overline{d}_n(x,y)=\frac{1}{n}\sum_{i=0}^{n-1}d(T^ix,T^iy)<\delta$ then it is easy to see
$$\frac{1}{n}\# \{ i\in [0,n-1]: d(T^ix,T^iy)\ge \sqrt{\delta}\}<\sqrt{\delta}$$
and so for $x,y\in A'=:A\setminus E_n$ (note that $\rho(A')>1-(\frac{\epsilon}{4})^2-\frac{\epsilon}{4}>1-\frac{\epsilon}{2}$).
\begin{align*}
 \overline{d'}_n(\phi(x),\phi(y))&= \frac{1}{n} \sum_{i=0}^{n-1} d'(S^i\phi(x),S^i\phi(y))=\frac{1}{n} \sum_{i=0}^{n-1} d'(\phi(T^i x),\phi(T^iy))\\
&\le \frac{1}{n} \big(\text{diam}(Y)\cdot \# \{ i\in [0,n-1]: d(T^ix,T^iy)\ge \sqrt{\delta}\}\\
&\hskip1cm+\frac{1}{n} \big(\text{diam}(Y)\cdot \# \{ i\in [0,n-1]: T^ix\not\in A,\ \text{or}\ T^iy\not\in A\} \\
&\hskip1cm+ \frac{\epsilon}{3}\# \{ i\in [0,n-1]: d(T^ix,T^iy)< \sqrt{\delta}\}\big)\ (\text{by} \ (\ref{hz-1}))\\
&\le \text{diam}(Y)\cdot \sqrt{\delta}+ \frac{\epsilon}{2} \cdot \text{diam}(Y) +  \frac{\epsilon}{3}<\epsilon.
\end{align*}
Pick $x_1,x_2,\cdots,x_m\in X$ such that $m=S_n(d,\rho,\frac{\delta}{2})$ and
$$\rho\big(\bigcup_{i=1}^m B_{\overline{d}_n}(x_i,\frac{\delta}{2})\big)>1-\frac{\delta}{2}.$$
Let $I_n=\{ r\in [1,m]: B_{\overline{d}_n}(x_r,\frac{\delta}{2})\cap A' \neq \emptyset\}$.
For $r\in I_n$, we choose $y_r^n\in B_{\overline{d}_n}(x_r,\frac{\delta}{2})\cap A'$. Then
$$\bigcup_{r\in I_n} \big( B_{\overline{d}_n}(y_r^n,\delta)\cap A'\big) \supseteq \bigcup_{r\in I_n} \big( B_{\overline{d}_n}(x_r,\frac{\delta}{2})\cap A' \big)=\big(\bigcup_{i=1}^m B_{\overline{d}_n}(x_i,\frac{\delta}{2})\big)\cap A'.$$
Thus
$$\rho(\bigcup_{r\in I_n} \big( B_{\overline{d}_n}(y_r^n,\delta)\cap A'\big))\ge \rho(\big(\bigcup_{i=1}^m B_{\overline{d}_n}(x_i,\frac{\delta}{2})\big)\cap A')
>1-\frac{\delta}{2}-(\frac{\epsilon}{4})^2-\frac{\epsilon}{4}>1-\epsilon.$$
Since $\overline{d'}_n(\phi(x),\phi(y))<\epsilon$ for $x,y \in A'$ with $\overline{d}_n(x,y) < \delta$,
ones has $$\phi(B_{\overline{d}_n}(y_r^n,\delta)\cap A')\subseteq B_{\overline{d'}_n}(\phi(y_r^n),\epsilon)$$ for $r\in I_n$.
Thus
\begin{align*}
\nu(\bigcup_{r\in I_n} B_{\overline{d'}_n}(\phi(y_r^n),\epsilon))&\ge
\nu(\bigcup_{r\in I_n} \phi(B_{\overline{d}_n}(y_r^n,\delta)\cap A'))\\
&=\nu (\phi \big(\bigcup_{r\in I_n}B_{\overline{d}_n}(y_r^n,\delta)\cap A'\big ))
=\rho(\bigcup_{r\in I_n}B_{\overline{d}_n}(y_r^n,\delta)\cap A')>1-\epsilon.
\end{align*}
Hence $S_n(d',\nu,\epsilon)\le |I_n|\le m=S_n(d,\rho,\frac{\delta}{2})$.

Finally, since the measure complexity of $(X,d,T,\rho)$  is weaker than $U(n)$, i,e.
$\liminf_{n\rightarrow +\infty} \frac{S_n(d,\rho,\frac{\delta}{2})}{U(n)}=0$, it deduces that $\liminf_{n\rightarrow +\infty} \frac{S_n(d',\nu,\epsilon)}{U(n)}=0$.
This finishes the proof of Proposition \ref{prop-0}.
\end{proof}

\begin{rem} From the proof we see that $\{S_n(d',\nu,\epsilon)\}_{n=1}^\infty$ is a bounded sequence if and only if so is $\{S_n(d,\rho,\epsilon)\}_{n=1}^\infty$.
\end{rem}
\subsection{The relation with other dynamical properties}
Let $h(T)$ denote the topological entropy of a t.d.s. $(X,T)$. The following result is similar
to the one given in \cite[Theorem 1.1]{Katok80}. We have

\begin{thm} Let $(X,T)$ be a t.d.s. Then, for any ergodic $\rho\in \mathcal{M}(X,T)$, we have
$$\lim_{\epsilon\rightarrow 0}\limsup_{n\rightarrow \infty}\frac{1}{n}\log S_n(d,\rho, \epsilon)= h_\rho(T)=\lim_{\epsilon\rightarrow 0}\liminf_{n\rightarrow \infty}\frac{1}{n}\log S_n(d,\rho, \epsilon).$$
Consequently, we have
$$\sup_{\rho\in \mathcal{M}(X,T)}\lim_{\epsilon\rightarrow 0}\limsup_{n\rightarrow \infty}\frac{1}{n}\log S_n(d,\rho, \epsilon)= h(T)=\sup_{\rho\in \mathcal{M}(X)}\lim_{\epsilon\rightarrow 0}\liminf_{n\rightarrow \infty}\frac{1}{n}\log S_n(d,\rho, \epsilon).$$
\end{thm}
\begin{proof} First we show for any $\rho\in \mathcal{M}(X,T)$
\begin{equation}\label{EE-1}
\limsup_{n\rightarrow \infty}\frac{1}{n}\log S_n(d,\rho, \epsilon)\le h(T) \text{ for any }\epsilon>0
\end{equation}
and when $\rho$ is ergodic,
\begin{equation}\label{EE-3}
\lim_{\epsilon\rightarrow 0}\limsup_{n\rightarrow \infty}\frac{1}{n}\log S_n(d,\rho, \epsilon)\le h_\rho(T).
\end{equation}

Fix $\epsilon>0$ and  $\rho\in \mathcal{M}(X,T)$. Define $d_n(x,y)=\max\limits_{0\le i\le n-1}d(T^ix,T^iy)$. It is clear that
$$\overline{d}_n(x,y)\le d_n(x,y)\ \text{and}\ B_{d_n}(x,\epsilon)\subset B_{\overline{d}_n}(x,\epsilon).$$

Let
$$r_n(d,\rho,\epsilon)=\min \{ m\in \mathbb{N}:\exists x_1,x_2,\cdots,x_m
\text{ s.t. } \rho\big(\bigcup_{i=1}^m B_{d_n}(x_i,\epsilon)\big)>1-\epsilon\},$$
see \cite[Section 7.2]{Wa} or \cite{Katok80} for the details. Then
there exist a finite subset $F$ of $X$ with $|F|=r_n(d,\rho,\epsilon/2)$ and $K\subset X$ with $\rho(K)>1-\epsilon/2$ such that
for any $x\in K$ there is $y\in F$ with $\overline{d}_n(x,y)\le d_n(x,y)\le \epsilon/2<\epsilon.$ This implies that
$$S_n(d,\rho, \epsilon)\le |F|=r_n(d, \rho, \epsilon/2),$$ and hence
$$\limsup_{n\rightarrow \infty}\frac{1}{n}\log S_n(d,\rho, \epsilon)\le \limsup_{n\rightarrow \infty}\frac{1}{n}\log r_n(d, \rho,\epsilon/2)\le h(T).$$

When $\rho$ is ergodic, by Theorem 1.1 of \cite{Katok80} we have
$$\lim_{\epsilon\rightarrow 0}\limsup_{n\rightarrow \infty}\frac{1}{n}\log S_n(d,\rho, \epsilon)\le \lim_{\epsilon\rightarrow 0}\limsup_{n\rightarrow \infty}\frac{1}{n}\log r_n(d, \rho,\epsilon/2)\le h_\rho(T).$$

Now we prove that for any $\rho\in \mathcal{M}(X,T)$
\begin{equation}\label{EE-2}
h_\rho(T)\le \lim_{\epsilon\rightarrow 0}\liminf_{n\rightarrow \infty}\frac{1}{n}\log S_n(d,\rho, \epsilon).
\end{equation}
Given  a Borel partition $\eta=\{A_1, \ldots, A_k\}$ of $X$ and $0<\delta<1$, we need to show
$$h_\rho(T,\eta)< a+2\delta,$$
where $a=:\lim_{\epsilon\rightarrow 0}\liminf_{n\rightarrow \infty}\frac{1}{n}\log S_n(d,\rho, \epsilon)$.

Take $0<\kappa<\frac{1}{2}$ with $$-2\kappa\log 2\kappa-(1-2\kappa)\log (1-2\kappa)+4\kappa \log (k+1) <\delta.$$
By Lemma 4.15 and Corollary 4.12.1 in \cite{Wa}, it is easy to see that there is a Borel partition $\xi=\{B_1, \ldots, B_k,B_{k+1}\}$ of $X$ associated with $\eta$ such that
$B_i\subset A_i$ is closed for $1\le i\le k$ and $B_{k+1}=X\setminus K$ with $K=:\cup_{i=1}^kB_i$, $\rho(B_{k+1})<\kappa^2$ and
\begin{align}\label{entropy-1}
h_\rho(T,\eta)\le h_\rho(T,\xi)+\delta.
\end{align}
It is clear that $B_i\cap B_j=\emptyset, 1\le i<j\le k$ and $\rho(K)>1-\kappa^2$.
Let $b=\min_{1\le i<j\le k} d(B_i,B_j)$. Then $b>0$.

Let $0<\epsilon<\min\{ 1-2\kappa, \frac{1}{2}b(1-2\kappa)\}$ and $n\in\mathbb{N}$. By the definition, there are $x_1,\ldots, x_{m(n)}\in X$ such that $\rho\big(\bigcup_{i=1}^{m(n)} B_{\overline{d}_n}(x_i,\epsilon)\big)>1-\epsilon$, where
$m(n)=S_n(d,\rho, \epsilon).$ Let $F_n=\bigcup_{i=1}^{m(n)} B_{\overline{d}_n}(x_i,\epsilon).$

For  $x\in X$, let $E(x)=\{i\ge 0: T^ix\in K\}$ and
$ E_n=\{x\in X:\frac{|E(x)\cap [0,n-1]|}{n}\le 1-\kappa\}.$
By Lemma \ref{severaltimes} we have
$\rho(E_n)<\kappa$. Put $W_n=(K\cap F_n)\setminus E_n$. Then
$$\rho(W_n)>1-\kappa^2-\kappa-\epsilon>1-2\kappa-\epsilon$$
and for $z\in W_n$,
\begin{equation*}
\frac{|E(z)\cap [0,n-1]|}{n}> 1-\kappa.
\end{equation*}

%
For a given $1\le i\le m$ we claim
\begin{align}\label{entropy-2}
|\{A\in \xi_0^{n-1}: A\cap B_{\overline{d}_n}(x_i,\epsilon)\cap W_n \not=\emptyset\}|\le C_n^{[nc(\epsilon)]} \cdot (k+1)^{[nc(\epsilon)]},
\end{align}
where $c(\epsilon)=2\frac{\epsilon}{b}+2\kappa$ and $[nc(\epsilon)]$ is the integer part of $nc(\epsilon)$.
To see this, let $x\in A_1\cap B_{\overline{d}_n}(x_i,\epsilon)\cap W_n \not=\emptyset$ and $y\in A_2\cap B_{\overline{d}_n}(x_i,\epsilon)\cap W_n \not=\emptyset$, where $A_1,A_2\in \xi_0^{n-1}$. Then $\overline{d}_n(x,y)<2\epsilon$ and $|E(x)\cap E(y)\cap [0,n-1]|\ge (1-2\kappa) n$. This implies that
$$ |\{0\le i\le n-1: T^ix\in B_{k_i},T^iy\in B_{l_i}, 1\le k_i\not=l_i\le k+1\}|\le [nc(\epsilon)].$$

It remains to show $h_\rho(T,\xi)\le a+\delta$. We know that
\begin{align*}
H_\rho(\xi_0^{n-1})&\le
H_\rho(\xi_0^{n-1}\vee \{ W_n,X\setminus W_n\})\\
&\le \rho(W_n) \log |\{ A\in \xi_0^{n-1}: A \cap W_n\neq \emptyset\}|+\rho(X\setminus W_n)\log |\xi_0^{n-1}|\\
&\hskip0.8cm -\rho(W_n)\log \rho(W_n)-(1-\rho(W_n))\log (1-\rho(W_n))\\
&\le \log |\{ A\in \xi_0^{n-1}: A \cap W_n\neq \emptyset\}|+n\rho(X\setminus W_n)\log (k+1)+\log 2\\
&\le \log \big( \sum_{i=1}^{m(n)} |\{A\in \xi_0^{n-1}: A\cap B_{d_n}(x_i,\epsilon)\cap W_n \not=\emptyset\}|\big)\\
&\hskip0.8cm +(2\kappa+\epsilon)n \log (k+1)+\log 2.
\end{align*}
See \cite[Section 8.1]{Wa} for some inequality in the above estimation.  Combining this with \eqref{entropy-2}, we have
$$H_\rho(\xi_0^{n-1})\le \log \big( S_n(d,\rho,\epsilon) C_n^{[nc(\epsilon)]}(k+1)^{[nc(\epsilon)]}\big)+(2\kappa+\epsilon)n \log (k+1)+\log 2. $$
Thus
\begin{align*}
h_\rho(T,\xi)&=\lim_{n\rightarrow +\infty} \frac{1}{n}H_\rho(\xi_0^{n-1})\\
&\le \liminf_{n\rightarrow +\infty} \frac{1}{n}\log S_n(d,\rho,\epsilon) +\limsup_{n\rightarrow +\infty} \frac{1}{n}\log C_n^{[nc(\epsilon)]}+(c(\epsilon)+2\kappa+\epsilon) \log (k+1)\\
&\le a+\limsup_{n\rightarrow +\infty} \frac{1}{n}\log C_n^{[nc(\epsilon)]}+(c(\epsilon)+2\kappa+\epsilon) \log (k+1).
\end{align*}
By Stirling's formula we have $$\lim_{n\rightarrow +\infty} \frac{1}{n}\log C_n^{[nc(\epsilon)]}=-c(\epsilon)\log c(\epsilon)-(1-c(\epsilon))\log (1-c(\epsilon))$$
Letting $\epsilon\rightarrow 0$  we obtain
$$h_\rho(T,\xi)\le a-2\kappa\log 2\kappa-(1-2\kappa)\log (1-2\kappa)+4\kappa \log (k+1) <a+\delta$$
and
hence $h_\rho(T,\eta)< a+2\delta$ by \eqref{entropy-1}.

So the theorem following from (\ref{EE-1}), (\ref{EE-2}) and the variational principle.
\end{proof}

\begin{rem} By the above theorem  when considering entropy we get the same value either using the balls defined by
$\overline{d}_n$ or using the standard Bowen balls. The advantage to use $\overline{d}_n$ is that it is an isomorphic invariant (see
Proposition \ref{prop-0}) when studying the
complexity.
\end{rem}

\section{Proof of Theorem \ref{main-result1}}\label{section-proof}

In this section we will prove Theorem \ref{main-result1}.
First let us show that Theorem 1.1' implies Theorem \ref{main-result1}. To do this
let $(X,T)$ be a t.d.s. and the measure complexity of $(X,T,\rho)$  be sub-polynomial
for any $\rho\in {\mathcal{M}}(X,T)$. Assume the contrary that the  M\"{o}bius disjointness conjecture does not hold for this t.d.s.
Then there are $x\in X$, $N_1<N_2<\ldots$ and $f\in C(X)$ such that $\lim_{i\rightarrow +\infty} \frac{1}{N_i} \sum_{n=1}^{N_i} \mu(n) f(T^nx)=a$
with $a\not=0$. Without loss of generality assume that $\frac{1}{N_i}\sum_{n=1}^{N_i} \delta_{T^nx}$ weakly$^*$ converges to a $\rho\in  {\mathcal{M}}(X,T)$.
By the assumption the measure complexity of $\rho$ is sub-polynomial. Thus by Theorem 1.1' we have
$\lim_{i\rightarrow +\infty} \frac{1}{N_i} \sum_{n=1}^{N_i} \mu(n) f(T^nx)=0,$ a contradiction.

\medskip

Let $(X,T)$ be a t.d.s. with the metric $d$. Let $x\in X$ and $\{ N_1<N_2<N_3<\cdots\} \subseteq \mathbb{N}$ such that the sequence $\frac{1}{N_i}\sum_{n=1}^{N_i} \delta_{T^nx}$ weakly$^*$ converges to a Borel probability measure $\rho$.
Suppose that the measure complexity of $(X,d,T,\rho)$ is weaker than $U_\tau(n)=n^\tau$ for any $\tau>0$.

Now we  fix a real-valued function $f\in C(X)$. Without loss of generality, we assume that $\max_{x\in X}|f(x)|\le 1$. Fix $\epsilon\in (0,1)$.
To prove Theorem 1.1' it suffices to demonstrate
\begin{align}\label{mineq-1}
\limsup_{i\rightarrow +\infty} |\frac{1}{N_i}\sum_{n=1}^{N_i} \mu(n)f(T^nx)|<8\epsilon.
\end{align}
To this aim we show when $i$ is large enough in Lemma \ref{lem-2} that for some sequence $\{j_n\}$
$$|\frac{1}{N_i}\sum_{n=1}^{N_i} \mu(n)f(T^nx)-\frac{1}{N_i}\sum_{n=1}^{N_i} \Big( \frac{1}{L} \sum_{\ell=0}^{L-1} \mu(n+\ell) f(T^\ell x_{j_n}) \Big) |<5\epsilon.$$
and in Lemma \ref{lem-3} that
$$|\frac{1}{N_i}\sum_{n=1}^{N_i} \Big( \frac{1}{L} \sum_{\ell=0}^{L-1} \mu(n+\ell) f(T^\ell x_{j_n}) \Big) |<3\epsilon.$$
Note that $\{x_{j_n}\}\subset X$ will be defined in the process of the proof.

\medskip
One of the main tools in our proof is an estimate (Lemma \ref{lem-MRT-2} below) developed by
Matom\"{a}ki, Radziwi\l{}\l{} and Tao in \cite{MRT}. To explain
the result we need some preparation.
First we define the following sets $\mathcal{S}$ of natural numbers with typical prime factorization,
which were introduced in \cite{MR} (see also \cite[Definition 2.1]{MRT}).
\begin{de} Let $10<P_1<Q_1\le N$ and $\sqrt{N}\le N_0\le N$ be quantities such that $Q_1\le \exp(\sqrt{\log N_0})$. We then define
$P_j,Q_j$ for $j>1$ by the formula
$$P_j:= \exp(j^{4j}(\log Q_1)^{j-1} \log P_1) \text{ and } Q_j := \exp(j^{4j+2}(\log Q_1)^j).$$
for $j>1$. The intervals $[P_j, Q_j]$ are all disjoint and $P_j,Q_j\to\infty$. Take the largest number $J$ satisfying $Q_J \le \exp(\sqrt{\log N_0})$ and the
set $\mathcal{S}_{P_1,Q_1,N_0,N}$ of all the numbers $1 \le  n \le N$ which have at least one prime factor in the
$[P_j, Q_j]$ for each $1\le j \le J$.
\end{de}

$\mathcal{S}_{P_1,Q_1,N_0,N}$ only misses a small part of $\mathbb N$ when $Q_1$ is much larger than $P_1$.

\begin{lem}\label{lem-MRT-1} \cite[Lemma 2.2]{MRT}  Let $10<P_1<Q_1\le N$ and $\sqrt{N}\le N_0\le N$ such that $Q_1\le \exp(\sqrt{\log N_0})$.
Then, for every large enough $N$,
$$\# \{1\le n\le N: n\not\in \mathcal{S}_{P_1,Q_1,N_0,N}\}\le C \frac{\log P_1}{\log Q_1}\cdot N,$$
where $C>0$ is an absolute constant.
\end{lem}

Define a $1$-bounded
multiplicative function to be a multiplicative function $f : \mathbb{N}\rightarrow \mathbb{C}$ such that $|f(n)|\le  1$
for all $n \in \mathbb{N}$. Given two $1$-bounded multiplicative functions $f, g$ and a parameter
$N\ge 1$, we define the distance $\mathbb{D}(f, g;N) \in  [0,+\infty)$ by the formula
$$\mathbb{D}(f,g;N) := \big( \sum_{p\le N} \frac{1-Re(f(p)g(p))}{p} \big)
$$
It is known that this gives a (pseudo-)metric on $1$-bounded multiplicative functions; see \cite[Lemma 3.1]{GS}.
For any $1$-bounded multiplicative function $g$ and real number $N > 1$, let
$$M(g; N) := \inf_{|t|\le N} \mathbb{D}(g, n\mapsto n^{it}; N)^2$$
and
\begin{align*}
M(g;N,Q):&= \inf_{q\le Q; \chi \, (q)} M(g\overline{\chi};N)\\
&=\inf_{|t|\le N; q\le Q; \chi \, (q)} \mathbb{D}(g, n\mapsto \chi(n) n^{it};N)^2,
\end{align*}
where $\chi$ ranges over all Dirichlet characters of modulus $q\le Q$.

Following the discussion about the non-pretentiousness
of $\mu$ after Theorem 1.6 in \cite{MRT}, one has
\begin{align}\label{non-pre-mobious}
\lim_{N\rightarrow +\infty} M(\mu;N,Q)=+\infty
\end{align}
for any given $Q$. The following lemma is a direct application of \cite[Proposition 5.1]{MRT} to the M\"{o}bius function $\mu$.
\begin{lem}\label{lem-MRT-2} Let $10\le L,W,N$ be such that
$$ \log^{20} L \le W \le  \min\{ L^{1/500}, (\log N)^{1/125}\} \text{ and }W \le \exp(M(\mu;N,W)/3).$$
Set
$$\mathcal{S} = \mathcal{S}_{P_1,Q_1,\sqrt{10N},10N}$$
where $P_1:= W^{200}$, $Q_1:=L^{1/2}/W^3$.
Then
\begin{align}\label{MTR-ineq}
\frac{1}{NL^2}\sum_{\ell_1,\ell_2=0}^{L-1} \big|\sum_{n\in [1,N]\cap \mathcal{S}} \mu(n+\ell_1) \mu(n+\ell_2)\big|\le D \frac{1}{W^{1/20}}
\end{align}
where $D>0$ is an absolute constant.
\end{lem}

To establish the following two lemmas, we choose $\epsilon_1>0$ such that $\epsilon_1<\epsilon^2$ and
\begin{equation}\label{2017-3}
|f(y)-f(z)|<\epsilon\  \text{when}\ d(y,z)<\sqrt{\epsilon_1}.
\end{equation}
Put $\delta=\frac{\epsilon}{500(1+\epsilon)(1+10C)}$, where $C$ is the absolute constant in Lemma \ref{lem-MRT-1}. Then
\begin{align}\label{delta-choice}
0<\delta<\frac{1}{500} \text{ and }10C\frac{200\delta}{\frac{1}{2}-3\delta}<\epsilon.
\end{align}
Since the measure complexity of $(X,d,T,\rho)$  is sub-polynomial, there exists $L>10$ such that
\begin{equation}\label{2017-2}
m=S_L(d,\rho,\epsilon_1)<\frac{\epsilon^3 L^{\frac{\delta}{20}}}{2D}\text{ and } W\ge \max\{ 10,\log^{20} L \}.
\end{equation}
where $W:=L^\delta$ and $D$ is the absolute constant in Lemma \ref{lem-MRT-2}. Put
\begin{align}\label{p1q1}
P_1=W^{200}=L^{200\delta}\text{ and } Q_1=L^{\frac{1}{2}}/W^3=L^{\frac{1}{2}-3\delta}.
\end{align}

Then there exist $x_1,x_2,\cdots,x_m\in X$ such that
$$\rho\big(\bigcup_{i=1}^m B_{\overline{d}_L}(x_i,\epsilon_1)\big)>1-\epsilon_1>1-\epsilon^2.$$
Put $U=\bigcup_{i=1}^m B_{\overline{d}_L}(x_i,\epsilon_1)$ and $E=\{n\in \mathbb{N}:T^n x\in U\}$. Then $U$ is open and so
\begin{align}\label{ineq-0}
\liminf_{i\rightarrow +\infty}\frac{1}{N_i}|E\cap [1,N_i]|=\liminf_{i\rightarrow +\infty}\frac{1}{N_i}\sum_{n=1}^{N_i} \delta_{T^nx}(U)\ge \rho(U)>1-\epsilon_1.
\end{align}

For $n\in E$, we choose $j_n\in \{1,2,\cdots,m\}$ such that $T^nx\in B_{\overline{d}_L}(x_{j_n},\epsilon_1)$. Hence for $n\in E$, we have
$\overline{d}_L(T^nx,x_{j_n})<\epsilon_1$, i.e.
$$\frac{1}{L}\sum_{\ell=0}^{L-1} d\big( T^\ell(T^nx),T^\ell(x_{j_n})\big)<\epsilon_1$$
and so we have
\begin{equation}\label{2017-4}
\#\{\ell\in [0,L-1]:d(T^{\ell} (T^n x),T^{\ell}x_{j_n})\ge \sqrt{\epsilon_1}\}<L\sqrt{\epsilon_1}<L\epsilon.
\end{equation}
Thus for $n\in E$
\begin{align}\label{ineq-33}
&\hskip0.5cm \frac{1}{L} \sum_{\ell=0}^{L-1} |f(T^{\ell} (T^n x) )-f(T^\ell x_{j_n})|\nonumber\\
&\le \frac{1}{L} \big(\epsilon\#\{\ell\in [0,L-1]:d(T^{\ell} (T^n x),T^{\ell}x_{j_n})<\sqrt{\epsilon_1}\}\ \ (by\ (\ref{2017-3}))\\
&\hskip2cm+2\#\{\ell\in [0,L-1]:d(T^{\ell} (T^n x),T^{\ell}x_{j_n})\ge \sqrt{\epsilon_1}\}\big)<3\epsilon \nonumber,
\end{align}
by using the inequality (\ref{2017-4}) and the assumption $\max_{x\in X}|f(x)|\le 1$.

For each $n\notin E$, we simply set $j_n=1$.

\begin{lem}\label{lem-2} For all sufficiently large $i$,
$$\Big|\frac{1}{N_i}\sum_{n=1}^{N_i} \mu(n)f(T^nx)-\frac{1}{N_i}\sum_{n=1}^{N_i} \Big( \frac{1}{L} \sum_{\ell=0}^{L-1} \mu(n+\ell) f(T^\ell x_{j_n}) \Big)\Big|<5\epsilon.$$
\end{lem}
\begin{proof} As $\max_{x\in X}|f(x)|\le 1$, it is not hard to see that
 \begin{align}\label{ineq-1}
 \Big|\frac{1}{N_i}\sum_{n=1}^{N_i} \mu(n)f(T^nx)-\frac{1}{N_i}\sum_{n=1}^{N_i} \Big( \frac{1}{L} \sum_{\ell=0}^{L-1} \mu(n+\ell) f(T^{n+\ell} x) \Big)\Big|
 \le 2\frac{L}{N_i}.
 \end{align}

By \eqref{ineq-0} once $i$ is large enough,
 \begin{align}\label{ineq-2}
\frac{1}{N_i}|E\cap [1,N_i]|>1-\epsilon^2.
\end{align}
Now
\begin{align*}
&\hskip0.5cm \Big|\frac{1}{N_i}\sum_{n=1}^{N_i} \Big( \frac{1}{L} \sum_{\ell=0}^{L-1} \mu(n+\ell) f(T^{n+\ell} x) \Big)- \frac{1}{N_i}\sum_{n=1}^{N_i} \Big( \frac{1}{L} \sum_{\ell=0}^{L-1} \mu(n+\ell) f(T^{\ell} x_{j_n}) \Big)\Big|\\
&\le  \frac{1}{N_i}\sum_{n=1}^{N_i} \frac{1}{L} \sum_{\ell=0}^{L-1} |f(T^{\ell} (T^n x) )-f(T^\ell x_{j_n})|\\
&\le \frac{1}{N_i}\sum_{n\in [1,N_i]\setminus E} \frac{1}{L} \sum_{\ell=0}^{L-1} |f(T^{\ell} (T^n x) )-f(T^\ell x_{j_n})|\\
&\hskip3cm+\frac{1}{N_i}\sum_{n\in E\cap [1,N_i]} \frac{1}{L} \sum_{\ell=0}^{L-1} |f(T^{\ell} (T^n x) )-f(T^\ell x_{j_n})|\\
&<\frac{2}{N_i}|[1,N_i]\setminus E|+\frac{3\epsilon}{N_i}|E\cap [1,N_i]|     \ \ \ \text{ (by \eqref{ineq-33})}\\
&< \frac{2}{N_i}|[1,N_i]\setminus E|+3\epsilon.
\end{align*}

Combining this with \eqref{ineq-2}, when $i$ is large enough,
\begin{align}\label{ineq-4}
\begin{split}
&\hskip0.5cm \Big|\frac{1}{N_i}\sum_{n=1}^{N_i} \Big( \frac{1}{L} \sum_{\ell=0}^{L-1} \mu(n+\ell) f(T^{n+\ell} x) \Big)- \frac{1}{N_i}\sum_{n=1}^{N_i} \Big( \frac{1}{L} \sum_{\ell=0}^{L-1} \mu(n+\ell) f(T^{\ell} x_{j_n}) \Big)\Big| \\
&< 5\epsilon.
\end{split}
\end{align}
So the lemma follows by \eqref{ineq-4} and \eqref{ineq-1} when $i$ large enough as $\frac{L}{N_i}\rightarrow 0$.
\end{proof}

\begin{lem} \label{lem-3} For all sufficiently large $i$,
$$\Big|\frac{1}{N_i}\sum_{n=1}^{N_i} \Big( \frac{1}{L} \sum_{\ell=0}^{L-1} \mu(n+\ell) f(T^\ell x_{j_n}) \Big)\Big|<3\epsilon.$$
\end{lem}
\begin{proof} By \eqref{non-pre-mobious}, $\lim_{N\rightarrow +\infty} M(\mu;N,L)=+\infty$.
Hence there exists $i_0\in \mathbb{N}$ such that
$$W \le \exp(M(\mu;N_i,W)/3) \text{ and }L \le  \min\{ (\log N_i)^{1/125},\log \sqrt{10N_i}\}$$
for any $i\ge i_0$. It is clear that $$W\le L\le  (\log N_i)^{1/125}\text{ and }Q_1 \le L^{\frac{1}{2}}\le \exp(\sqrt{\log \sqrt{10N_i}})$$
when $i\ge i_0$.

Now fix an $i\ge i_0$. Let $\mathcal{S}_i=\mathcal{S}_{P_1,Q_1,\sqrt{10N_i},10N_i}$.
Then by Lemma \ref{lem-MRT-1} and the inequalities \eqref{delta-choice} and \eqref{p1q1}, for every $i\ge i_0$ (if necessary we enlarge $i_0$),
$$\# \{1\le n\le N_i: n\not\in \mathcal{S}_i \}\le 10C \frac{\log P_1}{\log Q_1}\cdot N_i<\epsilon N_i.$$

First for $y\in X$, let
$$E_i(y)=\{n\in [1,N_i]\cap \mathcal{S}_i: |\frac{1}{L} \sum_{\ell=0}^{L-1} \mu(n+\ell) f(T^\ell y)|\ge \epsilon\}.$$
Then
\begin{align}\label{2017-1}
\frac{|E_i(y)|}{N_i} \epsilon^2 &\le
\frac{1}{N_i}\sum_{n\in [1,N_i]\cap \mathcal{S}_i} \Big| \frac{1}{L} \sum_{\ell=0}^{L-1} \mu(n+\ell) f(T^\ell y) \Big |^2\nonumber\\
&=\frac{1}{N_i L^2}\sum_{n\in [1,N_i]\cap \mathcal{S}_i}  \sum_{\ell_1,\ell_2=0}^{L-1}\mu(n+\ell_1)\mu(n+\ell_2)f(T^{\ell_1}y)f(T^{\ell_2}y)\nonumber\\
&= \frac{1}{N_i L^2}  \sum_{\ell_1,\ell_2=0}^{L-1} \sum_{n\in [1,N_i]\cap \mathcal{S}_i} \mu(n+\ell_1)\mu(n+\ell_2)f(T^{\ell_1}y)f(T^{\ell_2}y)\nonumber\\
&\le \frac{1}{N_i L^2}  \sum_{\ell_1,\ell_2=0}^{L-1} |\sum_{n\in [1,N_i]\cap \mathcal{S}_i} \mu(n+\ell_1)\mu(n+\ell_2)|<\frac{D}{L^{\delta/20}},
\end{align}
where the last inequality follows from Lemma \ref{lem-MRT-2}.
Therefore
\begin{align*}
&\hskip0.5cm |\frac{1}{N_i}\sum_{n=1}^{N_i} \Big( \frac{1}{L} \sum_{\ell=0}^{L-1} \mu(n+\ell) f(T^\ell x_{j_n}) \Big) |\\
&\le \frac{1}{N_i}\big( \sum_{n\in [1,N_i]\cap \mathcal{S}_i} |\frac{1}{L} \sum_{\ell=0}^{L-1} \mu(n+\ell) f(T^\ell x_{j_n}) |+ \sum_{n\in [1,N_i]\setminus \mathcal{S}_i} |\frac{1}{L} \sum_{\ell=0}^{L-1} \mu(n+\ell) f(T^\ell x_{j_n}) |\big)\\
&\le \epsilon+\frac{1}{N_i}\sum_{n\in [1,N_i]\cap \mathcal{S}_i} |\frac{1}{L} \sum_{\ell=0}^{L-1} \mu(n+\ell) f(T^\ell x_{j_n}) |\\
&=\epsilon+\frac{1}{N_i}\big( \sum_{n \in \bigcup \limits_{j=1}^m E_i(x_j)}|\frac{1}{L} \sum_{\ell=0}^{L-1} \mu(n+\ell) f(T^\ell x_{j_n}) |+\\
&\hskip3cm \sum_{n \in [1,N_i]\cap \mathcal{S}_i\setminus \bigcup \limits_{j=1}^m E_i(x_j)}|\frac{1}{L} \sum_{\ell=0}^{L-1} \mu(n+\ell) f(T^\ell x_{j_n}) |\big)
\end{align*}

Since \begin{align*}
\frac{1}{N_i} \sum_{n \in \bigcup \limits_{j=1}^m E_i(x_j)}|\frac{1}{L} \sum_{\ell=0}^{L-1} \mu(n+\ell) f(T^\ell x_{j_n})|
&\le \frac{1}{N_i}\sum_{j=1}^m |E_i(x_j)|\\
&\le m\frac{D}{\epsilon^2L^{\delta/20}}<\frac{\epsilon^3}{2D} L^{\delta/20}\frac{D}{\epsilon^2L^{\delta/20}}
<\epsilon
\end{align*}
by (\ref{2017-1}) and (\ref{2017-2}), we have

\begin{align*}
|\frac{1}{N_i}\sum_{n=1}^{N_i} \Big( \frac{1}{L} \sum_{\ell=0}^{L-1} \mu(n+\ell) f(T^\ell x_{j_n}) \Big) |<2\epsilon+ \frac{1}{N_i} \sum_{n \in [1,N_i]\cap \mathcal{S}_i\setminus \bigcup \limits_{j=1}^m E_i(x_j)}\epsilon <3\epsilon
\end{align*}
as long as $i\ge i_0$ according to the definition of $E_i(y)$.
\end{proof}

\begin{proof}[Proof of Theorem 1.1'] It is clear that \eqref{mineq-1} follows by Lemma \ref{lem-2} and Lemma \ref{lem-3}.
This finishes the proof of Theorem 1.1'.
\end{proof}

\begin{rem} Using the above method it can be shown that if Chowla's conjecture holds for 2-terms, i.e.
$$\lim_{N\rightarrow \infty}\frac{1}{N}\sum_{n=1}^{N}\mu(n+h_1)\mu(n+h_2)=0$$
for any $0\le h_1<h_2\in \mathbb{N}$, and if a t.d.s. $(X,T)$ satisfies the condition that $\liminf_{n\rightarrow \infty} \frac{S_n(d,\rho,\epsilon)}{n}=0$
for any $\epsilon>0$ and $\rho\in \mathcal{M}(X,T)$, then the Sarnak's M\"{o}bius disjointness conjecture holds for $(X,T)$. One of our questions is the following
\begin{ques} Assume that the Sarnak's M\"{o}bius disjointness conjecture holds for any t.d.s. $(X,T)$ with the property that for any $\rho\in \mathcal{M}(X,T)$ there is a
$k\in\mathbb{N}$ such that $\liminf_{n\rightarrow \infty} \frac{S_n(d,\rho,\epsilon)}{n^k}=0$ for any $\epsilon>0$. Is it true that the Sarnak's M\"{o}bius disjointness conjecture holds
for any t.d.s. with zero entropy?
\end{ques}
\end{rem}

\section{Proof  of Theorem \ref{main-result2}}\label{section-proof1}
In this section, we prove Theorem \ref{main-result2}. Theorem \ref{main-result2} is a direct corollary
of the following Proposition \ref{prop-1} and Theorem \ref{main-result1}.

\begin{prop}\label{prop-1} Let $(X,T)$ be a t.d.s. and $\rho\in {\mathcal{M}}(X,T)$. If $\rho$ has discrete spectrum, then the measure complexity of $(X,T,\rho)$ is sub-polynomial.
\end{prop}

Indeed, we will show that for any given $\epsilon>0$, $S_n(d',\rho,\epsilon)$ is a bounded sequence for some compatible metric $d'$.

\begin{proof} Since $\rho$ has discrete spectrum, $L^2(\rho)$
is spanned by the set of eigenfunctions of $\rho$. Note that if $f$ is an eigenfunction of $\rho$, then
for any large enough $M\in \mathbb{N}$, $$f_M=\begin{cases} f(x) &\text{ if } |f(x)|\le M\\
0 &\text{ otherwise }
\end{cases}$$ is also an eigenfunction of $\rho$, and $\lim_{M\rightarrow +\infty} \|f_M-f\|_{L^2(\rho)}=0$.
Thus $L^2(\rho)$ is spanned by the set of bounded eigenfunctions of $\rho$.

Let $d$ be a metric on $X$ inducing the topology of $X$. Since $(X,d)$ is a compact metric space, $C(X)$ is separable.
Hence there exist a countable dense subset $\{ g_\ell \}_{\ell=1}^{\infty}$ of $C(X)$. We consider
the new metric
$$d'(x,y)=\sum_{\ell=1}^{+\infty} \frac{|g_\ell(x)-g_\ell(y)|}{2^\ell(2\|g_\ell \|+1)}$$
for any $x,y\in X$, where $\|g_\ell\|=\max_{x\in X} |g_\ell (x)|$. Clearly, $d'$ is a compatible metric with the topology of $X$.

In the following we are going to show that the measure complexity of $(X,d',T,\rho)$ is sub-polynomial.
Fix an $\epsilon>0$. It is sufficient to show that there exists $m=m(\epsilon)$ such that  $S_n(d',\rho,\epsilon)\le m$ for all $n\in \mathbb{N}$.

\medskip
\noindent {\bf Step 1: Define the number $m$}.

We first choose $L\ge \epsilon+2$ such that $\sum_{\ell=L+1}^{+\infty}\frac{1}{2^\ell}<\frac{\epsilon}{3}$.
Since $L^2(\rho)$ is spanned by the set of bounded eigenfunctions of $\rho$, we can find bounded eigenfunctions $\{ h_i\}_{i=1}^K$ of $\rho$
and $a_{\ell,k}\in \mathbb{C}$ for $\ell=1,2,\cdots,L$ and $k=1,2,\cdots,K$ such that
\begin{align}\label{markov-ineq}
\|g_\ell-\sum_{k=1}^K a_{\ell,k} h_k\|_{L^2(\rho)}<(\frac{\epsilon}{6L})^2
\end{align}
for $\ell=1,2,\cdots,L$.

For $k\in \{1,2,\cdots,K\}$, there exist $\lambda_k \in \mathbb{C}$ with $|\lambda_k|=1$, and $X_k\in \mathcal{B}_X$ with $\rho(X_k)=1$,
$TX_k\subseteq X_k$ such that $h_k(Tx)=\lambda_k h_k(x)$ for all $x\in X_k$.
Put
$$Y_\ell=\{x\in X: |g_\ell(x)-\sum_{k=1}^K a_{\ell,k} h_k(x)|\ge \frac{\epsilon}{6L}\}.$$
Then by  \eqref{markov-ineq}, $\rho(Y_\ell)<(\frac{\epsilon}{6L})^2$. Set $X_0=(\bigcap_{k=1}^K X_k)\setminus (\bigcup_{\ell=1}^L Y_\ell)$. Then
$$\rho(X_0)>1-L(\frac{\epsilon}{6L})^2=1-\frac{\epsilon^2}{36L}.$$
By Lusin's Theorem there exists a compact subset $A$ of $X_0$ such that $\rho(A)>1-\frac{\epsilon^2}{36L}$ and  $h_k|_A$ is a continuous function
for $k=1,2,\cdots,K$. Clearly
\begin{align}\label{cha-ineq}
|g_\ell(x)-\sum_{k=1}^K a_{\ell,k} h_k(x)|< \frac{\epsilon}{6L}
\end{align} for $\ell\in \{1,2,\cdots,L\}$ and $x\in A$.
By the continuity of $h_k|_A$, $k=1,2,\cdots,K$, there exists $\delta>0$ such that
when $x,y\in A$ with $d(x,y)<\delta$, one has
\begin{align}\label{cha-ineq1}
\sum_{\ell=1}^L \sum_{k=1}^K |a_{\ell,k}|\cdot |h_k(x)-h_k(y)|<\frac{\epsilon}{6}.
\end{align}

Since $A$ is compact, there exist $x_1,x_2,\cdots,x_m\in A$ such that
$\bigcup_{r=1}^m B_{d}(x_r,\frac{\delta}{3})\supseteq A$.

\medskip
\noindent {\bf Step 2: Show that $S_n(d',\rho,\epsilon) \le m$ for all  $n\in \mathbb{N}$.}

To do this for  $x\in X$, let $E(x)=\{i\ge 0: T^ix\in A\}.$
Then for $n\in \mathbb{N}$, let
$$F_n=\{x\in X:\frac{|E(x)\cap [0,n-1]|}{n}\le 1-\frac{\epsilon}{6}\}.$$
By the same proof of Lemma \ref{severaltimes}
we have $\rho(F_n)<\frac{\epsilon}{6L}$. Put $A_n=A\setminus F_n$. Then
$$\rho(A_n)>1-(\frac{\epsilon^2}{36L}+\frac{\epsilon}{6L})>1-\epsilon$$
and for $z\in A_n$,
\begin{equation}\label{esti-ineq}
\frac{|E(z)\cap [0,n-1]|}{n}> 1-\frac{\epsilon}{6}.
\end{equation}

Let $I_n=\{ r\in [1,m]: B_d(x_r,\frac{\delta}{3})\cap A_n\neq \emptyset\}$.
For $r\in I_n$, we choose $y_r^n\in B_d(x_r,\frac{\delta}{3})\cap A_n$. Then
$$\bigcup_{r\in I_n} \big( B_d(y_r^n,\delta)\cap A_n\big) \supseteq \bigcup_{r\in I_n} \big( B_d(x_r,\frac{\delta}{3})\cap A_n\big)=A_n.$$
Fix $r\in I_n$, for any $x\in B_d(y_r^n,\delta)\cap A_n$,
\begin{align*}
&\hskip0.5cm \overline{d'}_n(x,y_r^n)=\frac{1}{n}\sum_{i=0}^{n-1}d'(T^ix,T^iy_r^n)\\
&\le \frac{1}{n}\big(\sum_{i\in [0,n-1]\cap E(x)\cap E(y_r^n)}d'(T^ix,T^iy_r^n)
+|[0,n-1]\setminus E(x)|+|[0,n-1]\setminus E(y_r^n)|\big)\\
&< \frac{\epsilon}{3}+\frac{1}{n}\sum_{i\in [0,n-1]\cap E(x)\cap E(y_r^n)}d'(T^ix,T^iy_r^n)
\end{align*}
where the last inequality follows from \eqref{esti-ineq}.

Now we are going to estimate $d'(T^ix,T^iy_r^n)$.
For $i\in [0,n-1]\cap E(x)\cap E(y_r^n)$,
\begin{align*}
&\hskip0.5cm d'(T^ix,T^iy_r^n)\le \sum_{\ell=1}^L \frac{|g_\ell(T^ix)-g_\ell(T^iy_r^n)|}{2^\ell (2\|g_\ell\|+1)}+\sum_{\ell=L+1}^\infty \frac{1}{2^\ell}\\
&<\sum_{\ell=1}^L \frac{|g_\ell(T^ix)-g_\ell(T^iy_r^n)|}{2^\ell (2\|g_\ell\|+1)}+\frac{\epsilon}{3}\le\frac{\epsilon}{3}+\sum_{\ell=1}^L \frac{|g_\ell(T^ix)-g_\ell(T^iy_r^n)|}{2^\ell}\\
&<\frac{\epsilon}{3}+\sum_{\ell=1}^L \frac{1}{2^\ell} \big( |g_\ell(T^ix)-\sum_{k=1}^K a_{\ell,k} h_k(T^ix)|+|g_\ell(T^iy_r^n)-\sum_{k=1}^K a_{\ell,k} h_k(T^iy_r^n)|\\
&\hskip2cm+|\sum_{k=1}^K a_{\ell,k} (h_k(T^ix)-h_k(T^iy_r^n)|\big)\\
&<\frac{\epsilon}{3}+\sum_{\ell=1}^L \frac{1}{2^\ell} \big( 2\frac{\epsilon}{6L}+|\sum_{k=1}^K a_{\ell,k} (h_k(T^ix)-h_k(T^iy_r^n)|\big) \ \ \text{ (by \eqref{cha-ineq})}\\
&\le \frac{\epsilon}{3}+\frac{\epsilon}{6}+\sum_{\ell=1}^L \frac{1}{2^\ell} |\sum_{k=1}^K a_{\ell,k} (h_k(T^ix)-h_k(T^iy_r^n)|\\
&\le \frac{\epsilon}{2}+\sum_{\ell=1}^L \sum_{k=1}^K |a_{\ell,k}|\cdot|h_k(x)-h_k(y_r^n)|<\frac{2\epsilon}{3}.
\end{align*}
where the last inequality follows from \eqref{cha-ineq1}.

Combining the above two inequalities, one has
\begin{align*}
\overline{d'}_n(x,y_r^n)< \frac{\epsilon}{3}+\frac{1}{n}\sum_{i\in [0,n-1]\cap E(x)\cap E(y_r^n)}d'(T^ix,T^iy_r^n)
\le \epsilon.
\end{align*}
This implies $B_d(y_r^n,\delta)\cap A_n\subseteq B_{\overline{d'}_n}(y_r^n,\epsilon)$ for $r\in I_n$.

Summing up $\bigcup_{r\in I_n} B_{\overline{d'}_n}(y_r^n,\epsilon)\supseteq \bigcup_{r\in I_n}B_d(y_r^n,\delta)\cap A_n=A_n$ and
$\rho(A_n)>1-\epsilon$. Hence
$S_n(d',\rho,\epsilon)\le |I_n|\le m$.
This finishes the proof of Proposition \ref{prop-1}.
\end{proof}

We think that the following question has an affirmative answer.
\begin{ques}  \label{main-result2-con} Let $(X,T)$ be a t.d.s. and $\rho\in {\mathcal{M}}(X,T)$. Assume that
 $d$ is a metric on $X$ inducing the topology of $X$ such that for each $\epsilon>0$, $S_n(d,\rho, \epsilon)\le m=m(\epsilon)$
 for any $n\in\mathbb{N}$. Is it true that $\rho$ has discrete spectrum?
\end{ques}

\section{proof of Theorem \ref{thm-sub-polynomial}}
Let $T$ be a skew product map on $\mathbb{T}^2$ over an irrational rotation of the circle, i.e.
\begin{align*}
T(x, y) = (x+\alpha, y + h(x)),
\end{align*}
where $h : \mathbb{T}^1 \rightarrow \mathbb{T}^1$ is a homotopically trivial $C^\infty$-function and $\alpha\in [0,1)\setminus \mathbb{Q}$.

Fix a $\tau>0$ and  a $T$-invariant Borel probability measure $\rho$ on $\mathbb{T}^2$.
By Theorem \ref{main-result1} to show Theorem \ref{thm-sub-polynomial} it remains to prove that the measure
complexity of $(\mathbb{T}^2,T,\rho)$ is weaker than $U_\tau(n)=n^\tau$.

Since $h$ is homotopically trivial, $h$ can be realized as a $C^\infty$ function from $\mathbb{T}^1$ to $\mathbb{R}$ and be written as
\begin{align*}
h(x)=\sum_{m\in \mathbb{Z}} \widehat{h}(m)e(mx),
\end{align*}
where $e(\theta)=e^{2\pi i\theta}$, the convergence is uniform and the equality holds pointwise for all $x\in \mathbb{T}^1$.
Moreover, as $h$ is $C^\infty$, we have
\begin{align}\label{eq-esti-Fou}
|\widehat{h}(m)|\ll |m|^{-\tau_1}, \, \forall \, m\in \mathbb{Z}\setminus \{0\}
\end{align}
where $\tau_1=\frac{2}{\tau}+6$.

Consider the continued fraction expansion
$$\alpha=[0;a_1,a_2,\cdots]=\frac{1}{a_1+\frac{1}{a_2+\frac{1}{a_3+\cdots}}}$$
and let $\frac{p_k}{q_k}=[0;a_1,\cdots,a_{k-1}]$ be the corresponding convergents. As $\alpha$ is irrational, the expansion is infinite.
Let $\|\theta\|=\min_{k\in \mathbb{Z}}|\theta-k|$ for $\theta\in \mathbb{R}$. Remark that
$$\|\theta\|\ll |1-e(\theta)|\ll \|\theta\|.$$
\begin{rem}\label{rem-1} The following standard facts can be found in \cite{Khin97}:
\begin{enumerate}
\item $p_1=0,q_1=1$; $p_2=1,q_2=a_1$; $p_{k+1}=a_kp_k+p_{k-1}$ and $q_{k+1}=a_kq_k+q_{k-1}$ for $k\ge 2$;

\item $p_k$ is coprime to $q_k$;

\item $\frac{1}{q_{k+1}+q_k}<\|q_k\alpha\|<\frac{1}{q_{k+1}}$;

\item If $|\alpha-\frac{p}{q}|<\frac{1}{2q^2}$ for $p,q\in \mathbb{Z}$ ($q\neq 0$), then $\frac{p}{q}$
coincides with one of the $\frac{p_k}{q_k}$'s;

\end{enumerate}
\end{rem}
By Remark \ref{rem-1},
$$q_{k+1}\ge q_{k}+q_{k-1}\ge 2q_{k-1}.$$
Thus, $q_k$ grows exponentially:
\begin{align}\label{g-ex}
q_k\gg 2^{\frac{k}{2}} \text{ and }q_{k+j}\gg 2^{\frac{j}{2}}q_k.
\end{align}
Set
$$M=\bigcup_{k\in E}\{\pm m_kq_k:m_k=1,\cdots,a_k \},$$ where $E=:\{k\in\mathbb{N}:q_{k+1}>q_k^{\frac{1}{\tau}+3},\, k>1\}.$
We will show Theorem \ref{thm-sub-polynomial} by considering $M$ is finite or infinite.

\medskip

The following lemma is essentially contained in \cite[Lemma 4.1]{LS}, with slightly finer
estimates.
\begin{lem}\label{lem-M-conjugate} Assume \eqref{eq-esti-Fou} holds,
then $$\sum_{m\not \in M\cup\{0\}} \widehat{h}(m)\frac{1}{e(m\alpha)-1}e(mx)$$
converges uniformly to a continuous function $\psi(x)$.
\end{lem}
\begin{proof} Since $m \not \in M \cup \{0\}$, we are in one of the following two situations below:


\medskip
(1) For some $k$, $q_k \le |m| < q_{k+1}$ but $q_k\nmid  |m|$. In this case we claim $\|m\alpha\| \ge \frac{1}{2|m|}$.
To show the claim assume the contrary that  $\|m\alpha\| < \frac{1}{2|m|}$. By Remark \ref{rem-1} (4), $|m| = aq_j$ and
$\|m\alpha\| = |m\alpha-ap_j |$ for
some index $j\le k$ and $a\in \mathbb{N}$. Since $q_k \nmid |m|$, $j < k$. Hence we have
\begin{align*}
\|m\alpha\| = |a| \cdot \|q_j\alpha\|>\frac{|a|}{q_{j+1} + q_j}\ge \frac{1}{2q_k}\ge \frac{1}{2|m|},
\end{align*}
a contradiction.

Therefore for any given $k$
\begin{align}\label{ineq-lmc-1}
\begin{split}
&\hskip0.5cm \sum_{ q_k\le |m|<q_{k+1}\atop
q_k\nmid m}| \widehat{h}(m)\frac{1}{e(m\alpha)-1}e(mx)|\\
&\ll \sum_{ q_k\le |m|<q_{k+1}\atop  qk\nmid m} (|m|^{-\tau_1}\cdot |m|\cdot 1)\ll \sum_{ m=q_k}^{q_{k+1}} m^{-(\tau_1-1)} \\
&\ll (q_k^{-(\tau_1-2)}-q_{k+1}^{-(\tau_1-2)}),
\end{split}
\end{align}
since $||m\alpha||\ll |e(m\alpha)-1|$ and $\|m\alpha\| \ge \frac{1}{2|m|}$.

\medskip

(2) $m = \pm m_kq_k$, where $m_k \in \{1,\cdots, a_k\}$ but $q_{k+1}\le   q_k^{\frac{1}{\tau}+3}$. Since
$$m_k\|q_k\alpha \|\le  \frac{a_k}{q_{k+1}}<\frac{1}{q_k},\ \text{by Remark \ref{rem-1}}$$
$\|m\alpha\|$ is given by $m_k \|q_k\alpha\|$ for $k\ge 3$. Note that by Remark \ref{rem-1}  $m_k \|q_k\alpha\|>m_k \frac{1}{q_k+q_{k+1}}.$
Thus, we have for all $k\ge 3$ that
\begin{align}\label{ineq-lmc-2}
\begin{split}
&\hskip0.5cm \sum_{ q_k\le |m|<q_{k+1}\atop
q_k\mid m}| \widehat{h}(m)\frac{1}{e(m\alpha)-1}e(mx)|\\
&\ll 2 \sum_{m_k=1}^{a_k}  ((m_k q_k)^{-\tau_1} \cdot \frac{1}{m_k\cdot \frac{1}{q_k+q_{k+1}} } \cdot 1) \\
&\ll \sum_{m_k=1}^{a_k}  m_k^{-(\tau_1+1)} \cdot q_{k+1}q_k^{-\tau_1}\ll \sum_{m_k=1}^{a_k}  m_k^{-(\tau_1+1)}\cdot q_k^{-(\frac{1}{\tau}+3)} \\
&\ll q_k^{-(\frac{1}{\tau}+3)}.
\end{split}
\end{align}
The last inequality follows from the fact that $\sum_{m_k=1}^{+\infty}  m_k^{-(\tau_1+1)}<\sum_{m_k=1}^{+\infty}  m_k^{-2}<2$.

Now we sum up both estimates \eqref{ineq-lmc-1} and \eqref{ineq-lmc-2} over all $k\ge 3$. Since only finitely
many terms are neglected in doing this, and the estimates are independent
of $x$, to prove the lemma it suffices to know that both the resulting series
are convergent. This is indeed the case, respectively since $\tau_1-2 > 0$ and
$\frac{1}{\tau}+3> 1$. This ends the proof of the lemma.
\end{proof}

Using the above lemma we are able to study the case when $M$ is finite.

\begin{prop}\label{lem-M-finite} Assume \eqref{eq-esti-Fou} holds and $M$ is finite, then the measure complexity of $(\mathbb{T}^2,T,\rho)$ is weaker than $U_\tau(n)=n^\tau$.
\end{prop}
\begin{proof} Since $M$ is finite, the function
$$\sum_{m\neq 0} \widehat{h}(m)\frac{1}{e(m\alpha)-1}e(mx)$$
differs from $\psi$ in Lemma \ref{lem-M-conjugate} by only finitely many terms and also converges uniformly to a continuous function $\widetilde{\psi}(x)$. Let $$\widetilde{S}(x,y)=(x+\alpha,y+\widehat{h}(0))$$
for $(x,y)\in \mathbb{T}^2$.
Put $\widetilde{\pi}:\mathbb{T}^2\rightarrow \mathbb{T}^2$ with
$\widetilde{\pi}(x,y)=(x,y-\widetilde{\psi}(x))$. Let $\widetilde{\nu}=\rho\circ \widetilde{\pi}^{-1}$.
Then $(\mathbb{T}^2,T,\rho)$ is measurably isomorphic to $(\mathbb{T}^2,\widetilde{S},\widetilde{\nu})$ by $\widetilde{\pi}$. Now $\widetilde{S}$ is a rotation of  $\mathbb{T}^2$.
We endow a rotation-invariant metric $d$ on $\mathbb{T}^2$. Then  for $\epsilon >0$,
$$S_n(d,\widetilde{\nu},\epsilon)= S_1(d,\widetilde{\nu},\epsilon)<\infty$$
for all $n\in \mathbb{N}$. Thus the measure complexity of $(\mathbb{T}^2,\widetilde{S},\widetilde{\nu})$ is weaker than $U_\tau(n)=n^\tau$. By Proposition \ref{prop-0}, the measure complexity of $(\mathbb{T}^2,T,\rho)$ is also
weaker than $U_\tau(n)=n^\tau$.
\end{proof}

Let $h_1(x)=\sum_{m \in M\cup\{0\}} \widehat{h}(m)e(mx)$ and define $S:\mathbb{T}^2\rightarrow \mathbb{T}^2$ such that
$$S(x,y)=(x+\alpha,y+h_1(x))$$
for $(x,y)\in \mathbb{T}^2$. Note that $h_1$ is $C^\infty$.
Define $\pi:\mathbb{T}^2\rightarrow \mathbb{T}^2$ with
$\pi(x,y)=(x,y-\psi(x))$. Let $\nu=\rho\circ \pi^{-1}$.
Then $(\mathbb{T}^2,S)$ is a T.D.S., and
\begin{align*}
\pi\circ T(x,y)&=\pi(x+\alpha,y+h(x))=(x+\alpha,y+h(x)-\psi(x+\alpha))\\
&=S(x,y+h(x)-h_1(x)-\psi(x+\alpha))\\
&=S(x,y-\psi(x))=S\circ \pi(x,y)
\end{align*}
for any $(x,y)\in \mathbb{T}^2$ and
$\nu$ is a $S$-invariant Borel probability measure on $\mathbb{T}^2$.
Thus $(\mathbb{T}^2,T,\rho)$ is measurably isomorphic to $(\mathbb{T}^2,S,\nu)$ by $\pi$.

In the following we will consider the case when $M$ is infinite. By Proposition \ref{prop-0} we only need to study the measure complexity
of $(\mathbb{T}^2,S,\nu)$. To this aim, for $n\in \mathbb{N}$, let
$$H_n(x):=\sum_{i=0}^{n-1} h_1(x+i\alpha)$$
for $x\in \mathbb{T}^1$. Clearly
$$S^n(x,y)=(x+n\alpha,y+H_n(x))$$
for $(x,y)\in \mathbb{T}^2$ and $n\ge 0$, where
$H_0(x)\equiv 0$.
\begin{lem} \label{Lem-54} Assume \eqref{eq-esti-Fou} holds and $M$ is infinite, then
$$\max_{x\in \mathbb{T}^1} |H_{q_t}(x)-q_t\widehat{h}(0)|\ll q_t^{-(\frac{1}{\tau}+2)}$$
for  $t\in E:=\{ k\in \mathbb{N}:q_{k+1}>q_k^{\frac{1}{\tau}+3}, k>1\}$.
\end{lem}
\begin{proof} Fix $t\in E$. For $x\in \mathbb{T}^1$,
\begin{align*}
&\hskip0.5cm |H_{q_t}(x)-q_t\widehat{h}(0)|\\
&=|\sum_{k\in E} \sum_{\substack{-a_k\le j\le a_k\\
 j\neq 0}} \widehat{h}(jq_k) e(jq_kx)\cdot \frac{e(jq_kq_t\alpha)-1}{e(jq_k\alpha)-1}|\\
&\ll\sum_{k\in E} \sum_{\substack{-a_k\le j\le a_k\\ j\neq 0}} |jq_k|^{-\tau_1}  \frac{\|jq_kq_t\alpha\|}{\|jq_k\alpha\|} \ll 2\sum_{k\in E} \sum_{j=1}^{a_k} |jq_k|^{-\tau_1}  \frac{\|jq_kq_t\alpha\|}{\|jq_k\alpha\|}\\
&\ll \big(\sum_{k\in E\cap [2,t-1]}\sum_{j=1}^{a_k} |jq_k|^{-\tau_1}  \frac{\|jq_kq_t\alpha\|}{\|jq_k\alpha\|}\big)+\big(\sum_{k\in E\cap [t,+\infty)}\sum_{j=1}^{a_k} |jq_k|^{-\tau_1}  \frac{\|jq_kq_t\alpha\|}{\|jq_k\alpha\|}\big)\\
&\ll (\sum_{k\in E\cap [2,t-1]}\sum_{j=1}^{a_k} |jq_k|^{-\tau_1}  \frac{jq_k\|q_t\alpha\|}{\frac{j}{q_{k+1}+q_k}})+\big(\sum_{k\in E\cap [t,+\infty)}\sum_{j=1}^{a_k} |jq_k|^{-\tau_1}  q_t \big)\\
&\ll 2q_t \|q_t\alpha\|( \sum_{k\in E\cap [2,t-1]}\sum_{j=1}^{a_k} |jq_k|^{-(\tau_1-1)} )+q_t^{-(\tau_1-3)} \big(\sum_{k\in E\cap [t,+\infty)}\sum_{j=1}^{a_k} |jq_k|^{-2}\big)\\
&\ll (2q_t\|q_t\alpha\|+q_t^{-(\tau_1-3)}) (\sum_{n=1}^{+\infty}\frac{1}{n^2})\ll \frac{q_t}{q_{t+1}}+q_t^{-(\tau_1-3)}\\
&\ll q_t^{-(\frac{1}{\tau}+2)},
\end{align*}
where the last inequality follows from the fact $q_{t+1}>q_t^{ \frac{1}{\tau} +3}$.
\end{proof}

Now we are ready to study the case when $M$ is infinite.
\begin{prop}\label{lem-M-infinite} Assume \eqref{eq-esti-Fou} holds and $M$ is infinite, then the measure complexity of $(\mathbb{T}^2,T,\rho)$ is weaker than $U_\tau(n)=n^\tau$.
\end{prop}
\begin{proof} By Proposition \ref{prop-0}, it is sufficient to prove that the measure complexity of $(\mathbb{T}^2,S,\nu)$ is weaker than $U_\tau(n)=n^\tau$.  That is, it is sufficient to prove that
\begin{align}\label{eq-2222}
\liminf_{n\rightarrow +\infty} \frac{S_n(d,\nu,\epsilon)}{n^\tau}=0
\end{align}
for any $\epsilon>0$, where the metric $d$ is defined by
$$d((x_1,y_1),(x_2,y_2)):=\max \{ \|x_1-x_2\|,\|y_1-y_2\|\}$$
for any $(x_1,y_1),(x_2,y_2)\in \mathbb{T}^1\times \mathbb{T}^1$.

Let $\epsilon>0$. By Lemma \ref{Lem-54}, there exists a positive constant $C$ such that
\begin{align*}
\max_{x\in \mathbb{T}^1} |H_{q_t}(x)-q_t\widehat{h}(0)|\le C q_t^{-(\frac{1}{\tau}+2)}
\end{align*}
for  all $t\in E:=\{ k\in \mathbb{N}: q_{k+1}>q_k^{\frac{1}{\tau}+3}, k>1\}$.
This implies that
\begin{align}\label{ineq-control}
\max_{x_1,x_2\in \mathbb{T}^1} |H_{q_t}(x_1)-H_{q_t}(x_2)|\le 2C q_t^{-(\frac{1}{\tau}+2)}
\end{align}
for  all $t\in E$.

Choose $t_0\in \mathbb{N}$ such that $\frac{2C}{q_t}<\frac{\epsilon}{3}$ for $t\ge t_0$.
For $t\in E\cap [t_0,+\infty)$, let $n_t=q_t^{[\frac{1}{\tau}]+2}$, where $[\frac{1}{\tau}]$ is the integer part of $\frac{1}{\tau}$.

Note that $h_1(x)=\sum_{m \in M\cup\{0\}} \widehat{h}(m)e(mx)$ is $C^\infty$. Hence there exists  $L\in \mathbb{N}$ such that
$L\ge \frac{3}{\epsilon}$ and
\begin{align}\label{ineq-control1}
|h_1(x_1)-h_1(x_2)|\le L\|x_1-x_2\|
\end{align}
for any $x_1,x_2\in \mathbb{T}^1$.

Now we are going to show that for $t\in E\cap [t_0,+\infty)$,
\begin{align}\label{esti-333}
S_{n_t}(d,\nu,\epsilon)\le L^2q_t([\frac{3}{\epsilon}]+1).
\end{align}
Given $t\in E\cap [t_0,+\infty)$, let
$$F_t=\{ (\frac{i}{Lq_t([\frac{3}{\epsilon}]+1)},\frac{j}{L}): (i,j)\in \{0,1,\cdots,Lq_t([\frac{3}{\epsilon}]+1)-1\}\times \{0,1,\cdots,L-1\}\}.$$
Then for any $(x,y)\in \mathbb{T}^2$, we can find
$(x_*,y_*)\in F_t$ such that
\begin{align}\label{ineq-control2}
\|x-x_*\|\le \frac{1}{Lq_t([\frac{3}{\epsilon}]+1)} \text{ and }\|y-y_*\|\le \frac{1}{L}\le \frac{\epsilon}{3}.
\end{align}
For $i\in \{0,1,\cdots,n_t-1\}$, write
$$i=a_iq_t+b_i, \text{ where }a_i\in \{0,1,\cdots,q_t^{[\frac{1}{\tau}]+1}-1\}, \, 0\le b_i\le q_t-1.$$
Then
\begin{align*}
&S^i(x,y)=(x+i\alpha,y+\sum_{r=0}^{a_i-1}H_{q_t}(x+rq_t)+\sum_{j=0}^{b_i}h_1(x+(a_iq_t+j)\alpha),\\
&S^i(x_*,y_*)=(x_*+i\alpha,y_*+\sum_{r=0}^{a_i-1}H_{q_t}(x_*+rq_t)+\sum_{j=0}^{b_i}h_1(x_*+(a_iq_t+j)\alpha).
\end{align*}
Thus by \eqref{ineq-control}, \eqref{ineq-control1} and \eqref{ineq-control2}, we have
\begin{align*}
&\hskip0.5cm d(S^i(x,y),S^i(x_*,y_*))\\
&=\max\{ \|x-x_*\|, \|y-y_*+\sum_{r=0}^{a_i-1}\big(H_{q_t}(x+rq_t)-H_{q_t}(x_*+rq_t)\big)\\
&\hskip1.5cm +\sum_{j=0}^{b_i}\big(h_1(x+(a_iq_t+j)\alpha)-h_1(x_*+(a_iq_t+j)\alpha)\big)\|\}\\
&\le \max\{ \|x-x_*\|,\|y-y_*\|+a_i 2C q_t^{-(\frac{1}{\tau}+2)}+(b_i+1) L\|x-x_*\|\}\\
&< \frac{1}{L}+q_t^{[\frac{1}{\tau}]+1} 2C q_t^{-(\frac{1}{\tau}+2)}+Lq_t \frac{1}{Lq_t([\frac{3}{\epsilon}]+1)}\\
&\le \frac{1}{L}+\frac{2C}{q_t}+ \frac{1}{[\frac{3}{\epsilon}]+1}<\epsilon.
\end{align*}
It deduces that
$$\overline{d}_{n_t}((x,y),(x_*,y_*))=\frac{1}{n_t}\sum_{i=0}^{n_t-1}d(S^i(x,y),S^i(x_*,y_*))<\epsilon.$$
That is, $(x,y)\in B_{\overline{d}_{n_t}}((x_*,y_*),\epsilon)\subseteq \bigcup_{(x',y')\in F_t} B_{\overline{d}_{n_t}}((x',y'),\epsilon)$.
This implies $$\bigcup_{(x',y')\in F_t} B_{\overline{d}_{n_t}}((x',y'),\epsilon)=\mathbb{T}^2.$$
Hence $S_{n_t}(d,\nu,\epsilon)\le |F_t|=L^2q_t([\frac{3}{\epsilon}]+1)$. That is, \eqref{esti-333} holds.

\medskip

Next since $M$ is infinite, $E$ is also infinite. Thus using inequality \eqref{esti-333}, we have
\begin{align*}
\liminf_{n\rightarrow +\infty} \frac{S_n(d,\nu,\epsilon)}{n^\tau}&\le \liminf_{\substack{t\rightarrow +\infty\\ t\in E\cap [t_0,+\infty)}} \frac{S_{n_t}(d,\nu,\epsilon)}{n_t^\tau}\\
&\le  \liminf_{\substack{t\rightarrow +\infty\\ t\in E\cap [t_0,+\infty)}} \frac{L^2q_t([\frac{3}{\epsilon}]+1)}{q_t^{\tau([\frac{1}{\tau}]+2)}} \le\liminf_{\substack{t\rightarrow +\infty\\ t\in E\cap [t_0,+\infty)}}\frac{L^2([\frac{3}{\epsilon}]+1)}{q_t^{\tau}}=0.
\end{align*}
This implies that \eqref{eq-2222} holds and hence ends the proof.

\end{proof}

\begin{proof}[Proof of Theorem \ref{thm-sub-polynomial}] It is clear that the measure complexity of $(\mathbb{T}^2,T,\rho)$ is weaker than $U_\tau(n)=n^\tau$ by Proposition \ref{lem-M-finite} and Proposition \ref{lem-M-infinite} for any $\tau>0$. This finishes the proof of Theorem \ref{thm-sub-polynomial}.
\end{proof}

\section{Proof of Theorems \ref{thm-4} and \ref{thm-5}}
In this section, we  prove Theorems \ref{thm-4} and \ref{thm-5} by using Theorem \ref{main-result1}.
\subsection{Proof of Theorem \ref{thm-4}}
Firstly we can deduce Theorem \ref{thm-4}  from Theorem \ref{main-result1} as follows.
\begin{proof}[Proof of Theorem \ref{thm-4}] Now we shall fix a $T$-invariant Borel probability measure $\rho$ on $G\times \mathbb{T}^1$.
By Theorem \ref{main-result1}, it is sufficient to show that   the measure complexity of $(G\times \mathbb{T}^1,T,\rho)$ is sub-polynomial.

Let $\pi_G:G\times \mathbb{T}^1\rightarrow G$ be the coordinate projection. It is clear that $\rho\circ \pi_G^{-1}=m_G$ since $m_G$ is the unique $S_a$-invariant Borel probability of $G$. Since $T$  preserves
a measurable invariant section, there exists a Borel measurable map $\phi : G\rightarrow \mathbb{T}^1$ such that
$$T(g, \phi(g))=(ag,\phi(ag))$$
for $m_G$-a.e. every $g$. That is, $\phi(g)+h(g)=\phi(ag)$ for $m_G$-a.e. every $g$.
Define
$$\begin{cases} \pi(g,y)=(g,y-\phi(g)) \\
S(g,y)=(ag,y)
\end{cases}
\text{ for }(g,y)\in G\times \mathbb{T}^1.$$
Then $\pi:G\times \mathbb{T}^1\rightarrow G\times \mathbb{T}^1$ is an invertible Borel-measurable map and
$\pi^{-1}$ is also a Borel-measurable map.

Let $\nu=\rho \circ \pi^{-1}$ and $S:G\times \mathbb{T}^1\rightarrow G\times \mathbb{T}^1$ with $S(g,y)=(ag,y)$.
Then $(G\times \mathbb{T}^1,S)$ is a T.D.S. and $\nu$ is a Borel probability measure on $G\times \mathbb{T}^1$.
Note that for $m_G$-a.e. $g\in G$, $\pi\circ T(g,y)=S\circ \pi(g,y)$
for all $y\in Y$. Moreover since $\rho\circ \pi_G^{-1}=m_G$, one has
\begin{align}\label{cocycle-1}
\pi\circ T(g,y)=S\circ \pi(g,y)
\end{align}
for $\rho$-a.e. $(g,y)\in G\times \mathbb{T}^1$.

For any $F\in C(G\times \mathbb{T}^1)$,
\begin{align*}
\int_{G\times \mathbb{T}^1}F(S(g,y)) \ d \nu(g,y)&=\int_{G\times \mathbb{T}^1}F\big(S(\pi(g,y)\big) \ d \rho(g,y)\\
&=\int_{G\times \mathbb{T}^1}F\big(\pi(T(g,y))\big) \ d \rho(g,y)
\quad \ \text{ (by \ref{cocycle-1})}\\
&=\int_{G\times \mathbb{T}^1}F\big(\pi(g,y)\big) \ d \rho(g,y)=\int_{G\times \mathbb{T}^1}F(g,y) \ d \nu(g,y).
\end{align*}
This implies $\nu$ is $S$-invariant. Combining this with \eqref{cocycle-1},
$(G\times \mathbb{T}^1,T,\rho)$ is measurable isomorphic to $(G\times \mathbb{T}^1,S,\nu)$ by $\pi$.

Now $S$ is a rotation of the compact metric ableian group $G\times \mathbb{T}^1$.
We endow a rotation-invariant metric $d$ on $G\times \mathbb{T}^1$. Then  for $\epsilon >0$,
$$S_n(d,\nu,\epsilon)= S_1(d,\nu,\epsilon)<\infty$$
for all $n\in \mathbb{N}$. Thus the measure complexity of $(G\times \mathbb{T}^1,d,S,\nu)$ is sub-polynomial. By Proposition \ref{prop-0}, the measure complexity of $(G\times \mathbb{T}^1,T,\rho)$ is also sub-polynomial. Finally, M\"{o}bius disjointness conjecture holds for $(G\times \mathbb{T}^1,T)$
by Theorem \ref{main-result1}.
\end{proof}

\subsection{Proof of Theorem \ref{thm-5}} Theorem \ref{thm-5} is a direct corollary of the following
Proposition \ref{prop-2} and Theorem \ref{main-result2}.

\begin{prop}\label{prop-2} Let $f \in K(\mathbb{Z})$ and $\rho$ be a $T$-invariant Borel probability measure of $(X_f,T)$.
Then $\rho$ has discrete spectrum.
\end{prop}
\begin{proof} We follow the arguments in the proof of Theorem 5.2 in \cite{H}. To show that $(X_f,\mathcal{B}_{X_f},\rho, T )$ has discrete spectrum, it is sufficient to show that for any $g\in L^\infty (X_f,\mathcal{B}_{X_f},\rho)$, $\text{cl}\{U^n g : n\in \mathbb{Z}\}$ is compact
in $(L^2(\rho),\|\cdot\|_{L^2(\rho)})$, where $\|u\|_{L^2(\rho)}=(\int_{X_f}|u|^2 d \rho)^{\frac{1}{2}}$ for $u\in L^2(\rho)$.

Let $g\in L^\infty (X_f,\mathcal{B}_{X_f},\rho)$.
For any fixed sequence $\{h_i \}_{i\in \mathbb{N}}\subseteq \text{cl}\{U^n g : n\in \mathbb{Z}\}$,
pick $n_i\in \mathbb{Z}$ such that $\|U^{n_i} g-h_i\|_{L^2(\rho)}<\frac{1}{2i}$  for each $i\in \mathbb{N}$.
Since $f\in K(\mathbb{Z})$, $X_f$ is separable in the norm topology of $\ell^\infty(\mathbb{Z})$.
Thus there exist $x_m=(x_m(n))_{n\in \mathbb{Z}}\in X_f$, $m=1,2,\cdots$, such that for each $x=(x(n))_{n\in \mathbb{Z}} \in X_f$
\begin{equation}\label{separable}
\inf_{m\in \mathbb{N}} \|x_m-x\|_{\ell^\infty(\mathbb{Z})}=0,
\end{equation}
where $ \|x_m-x\|_{\ell^\infty(\mathbb{Z})}=\sup_{n\in \mathbb{Z}}|x_m(n)-x(n)|$.

Now we can find  a subsequence $i_1 < i_2 < \cdots$ of natural numbers such that
$\lim_{k\rightarrow +\infty} T^{n_{i_k}} (x_m)$ exists for each $m\in \mathbb{N}$.
By \eqref{separable} and the fact $$\|Ty-Tz\|_{\ell^\infty(\mathbb{Z})}=\|y-z\|_{\ell^\infty(\mathbb{Z})}$$
for any $y,z\in \ell^\infty(\mathbb{Z})$, it is not hard to see that
$\lim_{k\rightarrow +\infty} T^{n_{i_k}} (x)$ exists for each $x\in X_f$.

Define $p(x)=\lim_{k\rightarrow +\infty} T^{n_{i_k}} (x)$ for $x\in X_f$.
Clearly, $p$ is a Borel map from $X_f$ to $X_f$.
Let $h= g \circ p$. Then, $h\in  L^\infty (X_f,\mathcal{B}_{X_f},\rho)$. Since $\lim_{k\rightarrow +\infty} g (T^{ n_{i_k}} x) = h(x)$ for $\rho$-
almost every $x$ and $g, h\in L^\infty (X_f,\mathcal{B}_{X_f},\rho)$, it is not hard to see that $$\lim_{k\rightarrow +\infty} U^{n_{i_k}} g = h$$
in $(L^2(\rho),\|\cdot\|_{L^2(\rho)})$ and $h\in \text{cl}\{U^n g : n\in \mathbb{Z}\}$. Moreover,
\begin{align*}
&\hskip0.5cm \lim_{k\rightarrow +\infty} \|h_{i_k}- h\|_{L^2(\rho)}\\
&\le  \lim_{k\rightarrow +\infty} \big( \|h_{i_k}- U^{n_{i_k}} g\|_{L^2(\rho)}+\|U^{n_{i_k}} g- h\|_{L^2(\rho)}\big) \\
&�� \lim_{k\rightarrow +\infty} \big( \frac{1}{2i_k}+\|U^{n_{i_k}} g- h\|_{L^2(\rho)} \big)\\
&=0.
\end{align*}
That is, $\lim_{k\rightarrow +\infty} h_{i_k}=h$
in $(L^2(\rho),\|\cdot\|_{L^2(\rho)})$. This implies that the complete metric
space $\text{cl}\{U^n g : n\in \mathbb{Z}\}$ is sequential compact. Hence
$\text{cl}\{U^n g : n\in \mathbb{Z}\}$ is compact, since for a metric space, sequential compactness is equivalent
to compactness.
\end{proof}

\end{document}